\documentclass[12pt,reqno]{amsart}
\usepackage{amsfonts}

\usepackage{amssymb,amsmath,graphicx,amsfonts,euscript}
\usepackage{color}

\newtheorem{thm}{Theorem}[section]

\newtheorem{Pros}{Proposition}[section]
\newtheorem{lemma}{Lemma}[section]

\theoremstyle{definition}
\newtheorem{define}{Definition}[section]

\theoremstyle{remark}
\newtheorem{rem}{Remark}[section]

\allowdisplaybreaks \voffset=-0.2in \numberwithin{equation}{section}

\begin{document}
\bigskip\bigskip

\centerline{\Large\bf On the global regularity for anisotropic dissipative
}
\centerline{\Large\bf surface quasi-geostrophic equation
}

\bigskip

\centerline{Zhuan Ye}
\medskip

\medskip\smallskip

\centerline{ Department of Mathematics and Statistics, Jiangsu Normal University, }
\medskip

\centerline{101 Shanghai Road, Xuzhou 221116, Jiangsu, PR China}
\medskip
\centerline{E-mail: \texttt{yezhuan815@126.com
}}

\bigskip
{\bf Abstract:}~~%
In this paper, we consider the two-dimensional surface quasi-geostrophic equation with fractional horizontal dissipation and fractional vertical thermal diffusion. Global existence of classical solutions is established when the dissipation powers are restricted to a suitable range. Due to the nonlocality of these 1D fractional operators, some of the standard energy estimate techniques no longer apply, to overcome this difficulty, we establish several anisotropic embedding and interpolation inequalities involving fractional derivatives. In addition, in order to bypass the unavailability of the classical Gronwall inequality, we establish a new logarithmic type Gronwall inequality, which may be of independent interest and potential applications.

\vskip .1in
{ {\bf AMS Subject Classification 2010:}\quad 35A05; 35B45; 35B65; 76D03.

 }

\vskip .2in
\section{Introduction}

This paper concerns itself with the initial-value problem for the two-dimensional (2D)
surface quasi-geostrophic (abbr. SQG) equation with fractional horizontal dissipation and fractional vertical thermal
diffusion,  which
can be written as
\begin{equation}\label{SQG}
\left\{\aligned
&\partial_{t}\theta+(u \cdot \nabla)\theta+\mu\Lambda_{x_{1}}^{2\alpha}\theta+\nu\Lambda_{x_{2}}^{2\beta}\theta=0, \quad x=(x_{1},x_{2})\in \mathbb{R}^2, \,\, t>0,\\
&\theta(x, 0)=\theta_{0}(x),
\endaligned\right.
\end{equation}
where $\theta$ is a scalar real-valued function, $\mu\geq0,\,\nu\geq0,\, \alpha\in (0,\,1),\beta\in (0,\,1)$ are real constants, and the velocity ${u}\equiv (u_{1},\,u_{2})$ is determined by the Riesz transforms of the potential
temperature $\theta$ via the formula
$$ {u}=(u_{1},\,u_{2})=\left(-\frac{\partial_{x_{2}}}{ \sqrt{-\Delta}}\theta,\,\,\frac{\partial_{x_{1}}}{\sqrt{-\Delta}}\,\theta\right)
=(-\mathcal
{R}_{2}\theta,\,\,\mathcal {R}_{1}\theta):=\mathcal {R}^{\perp}\theta,$$
where $\mathcal{R}_{1}, \mathcal
{R}_{2}$ are the standard 2D Riesz transforms. Clearly, the velocity ${u}=(u_{1},\,u_{2})$ is divergence free, namely $\partial_{x_{1}} u_{1}+\partial_{x_{2}}u_{2}=0$.
The fractional operators $\Lambda_{x_{1}}:= \sqrt{-\partial_{x_{1}}^{2}}$ and $\Lambda_{x_{2}}:= \sqrt{-\partial_{x_{2}}^{2}}$ are defined through the Fourier transform, namely
$$
\widehat{\Lambda_{x_{1}}^{2\alpha}
f}(\xi)=|\xi_{1}|^{2\alpha}\hat{f}(\xi),\qquad \widehat{\Lambda_{x_{2}}^{2\beta}
f}(\xi)=|\xi_{2}|^{2\beta}\hat{f}(\xi),
$$
where
$$\hat{f}(\xi)=\frac{1}{(2\pi)^{2}}\int_{\mathbb{{R}}^{2}}{e^{-ix\cdot\xi}f(x)\,dx}.$$
The SQG equation arises from the geostrophic study of the highly rotating flow (see for instance \cite{PG}).
In particular, it is the special
case of the general quasi-geostrophic approximations for atmospheric and oceanic fluid flow with small Rossby and Ekman numbers, see \cite{CMT,PG} and the references cited there. Mathematically, as pointed out by Constantin, Majda and Tabak \cite{CMT}, the inviscid SQG equation (i.e., \eqref{SQG} with $\mu=\nu=0$) shares many parallel properties with those of the 3D Euler equations such as the vortex-stretching mechanism and thus
serves as a lower-dimensional model of the 3D Euler equations.
We remark that the inviscid SQG equation is probably among the simplest scalar partial differential equations, however, the global regularity problem still remains open.

\vskip .1in
The system (\ref{SQG}) is deeply related to the classical fractional dissipative SQG equation, with its form as follows
\begin{equation}\label{frSQG}
\left\{\aligned
&\partial_{t}\theta+(u \cdot \nabla)\theta+\mu\Lambda^{2\alpha}\theta=0, \\
&\theta(x, 0)=\theta_{0}(x),
\endaligned\right.
\end{equation}
where the classical fractional Laplacian operator $\Lambda^{2\alpha}:=(-\Delta)^{\alpha}$ is defined through the Fourier transform, namely
$$\widehat{\Lambda^{2\alpha}
f}(\xi)=|\xi|^{2\alpha}\hat{f}(\xi).$$
Obviously, the above system (\ref{frSQG}) can be deduced from the system (\ref{SQG}) with $\alpha=\beta$ and $\mu=\nu$. Because of its important physical background and profound mathematical significance,
the SQG equation attracts interest of scientists and mathematicians.
The first mathematical studies of the SQG equation was carried
out in 1994s by Constantin, Majda and Tabak \cite{CMT}, where they considered the inviscid SQG case, and established the local well-posedness and blow-up criterion in the Sobolev
spaces. Since then, the global regularity issue concerning the SQG has recently been studied very extensively and important progress has been made (one can see \cite{CCCGW} for a long list of references). Let us briefly recall some related works on the system (\ref{frSQG}).
Due to the battle between the orders of the
nonlinear term and the dissipation, the cases $\alpha>\frac{1}{2}$, $\alpha=\frac{1}{2}$ and $\alpha<\frac{1}{2}$ are called sub-critical, critical and super-critical, respectively.
The global regularity of the SQG equation seems to be in a satisfactory situation in the subcritical and critical cases.
The subcritical case has been essentially resolved in \cite{CW3111,Re} (see also \cite{DL,JUNING,SS} and references therein). Constantin, C$\rm \acute{o}$rdoba and Wu in \cite{CCW}
first addressed the global regularity issue for the critical case
and obtained a small data global existence result. More precisely, they showed that there is a unique
global solution when $\theta_{0}$ is in the critical space $H^{1}$ under a smallness assumption on $\|\theta_{0}\|_{L^{\infty}}$.
In fact, due to the balance of the nonlinear term and
the dissipative term in (\ref{frSQG}), the global existence of the critical case is a very
challenge issue, whose global regularity without small condition has
been successfully established by two elegant papers with totally
different approaches, namely Caffarelli-Vasseur \cite{CV1} via the De Giorgi iteration method and  Kiselev, Nazarov-Volberg \cite{KNV} relying on a new non local maximum principle. We also refer to  Kiselev-Nazarov\cite{KN2} and Constantin-Vicol \cite{CV} for another two delicate and still quite
different proofs of the same issue. See also the works \cite{Abidi1,DongD3,DongLi2,MX1} where same type of results have been obtained.
However, in terms of the supercritical case whether solutions (for large data)
remain globally regular or not is a remarkable open problem.
Although the global well-posedness for arbitrary initial data is still open for the supercritical SQG equation, some interesting regularity criteria (see for example \cite{CW3111,CCW111,DongLi2,Dongpa}) and small data global existence results (see for instance \cite{CMZ,CAD,Chae20031,JUNING,WangZ,WXY}) have been established.
Moreover, the global existence of weak solutions and the eventual regularity of the corresponding weak solutions to supercritical SQG equation have been
established (see, e.g. \cite{Re,Dab,Kis2011,Sil,ZV}).
For many other interesting results on the SQG equation, we refer to \cite{ConstantiW08,DongD3,SilVicolaz,CCCF05,W5}, just to mention a few.

\vskip .1in

As stated in the previous paragraph, on the one hand, it is not hard to establish the global regularity for the SQG equation (\ref{frSQG}) with $\alpha>\frac{1}{2}$. However, on the
other hand, the global regularity problem of the inviscid SQG equation is still an open problem.
Comparing these two extreme cases, it is natural for us to consider the intermediate cases.
Note that in all the papers mentioned above, the equation is assumed to have the standard fractional dissipation. In fact, compared with the SQG equation with the standard fractional dissipation, little has been done for the system (\ref{SQG}) as many techniques such as integration by parts no longer apply. Very recently, the author with collaborators in \cite{WXYjmfm} proved the global regularity result of the system (\ref{SQG}) with $\mu>0,\,\nu=0,\,\alpha=1$ or $\mu=0,\,\nu>0,\,\beta=1$.
In this paper, we consider the intermediate case to explore how fractional horizontal dissipation and fractional vertical thermal
diffusion would affect
the regularity of solutions to the SQG equation. To the best of
our knowledge, such system of equation as in (\ref{SQG}) has never been studied before. The main purpose of this paper is to
establish the global regularity when the dissipation powers are restricted to a suitable range. More specifically, the main result of this paper is the following global regularity result.
\begin{thm}\label{ThSQG}
Let $\theta_{0}\in H^{s}(\mathbb{R}^{2})$ for $s\geq2$. If $\alpha\in (0,\,1)$ and $\beta\in (0,\,1)$ satisfy
\begin{equation}\label{sdf2334}
\beta>\left\{\aligned
&\frac{1}{2\alpha+1},\qquad 0<\alpha\leq \frac{1}{2},\\
&\frac{1-\alpha}{2\alpha},\ \qquad \frac{1}{2}<\alpha<1,
\endaligned\right.
\end{equation}
then the system (\ref{SQG})
admits a unique global solution $\theta$ such that for any given $T>0$,
$$\theta\in C([0, T]; H^{s}(\mathbb{R}^{2})),\quad \Lambda_{x_{1}}^{\alpha}\theta,\ \Lambda_{x_{2}}^{\beta}\theta \in L^{2}([0, T]; H^{s}(\mathbb{R}^{2})).
$$
\end{thm}

\vskip .1in
We outline the main ideas and difficulties in the proof of this theorem.
Since the local well-posedness of (\ref{SQG}) follows from a standard procedure, a large portion of the efforts are devoted to obtaining global {\it a
priori} bounds for $\theta$ on the interval $[0,\,T]$ for any given $T>0$. For the sake of completeness, the local well-posedness part is presented in Appendix \ref{appd12}. The proof is largely divided into two steps, namely, the global $H^1$-estimate and the global $H^2$-estimate.
The first difficulty comes from the presence of the general 1D fractional Laplacian
dissipation which is a nonlocal operator, and thus some of the standard energy estimate techniques such as integration by parts no longer apply. Concerning the difficulty
caused by the presence of the 1D nonlocal operator, we need to establish the anisotropic embedding
and the interpolation inequalities involving fractional derivatives. The second major difficulty lies in the unboundedness of the Riesz transform between the space $L^{\infty}$. More precisely, if one tries to establish the global $H^1$-estimate, then one needs to control the quantity $\|u(t)\|_{L_{x}^{\infty}}$. However, due to the relation $u=\mathcal {R}^{\perp}\theta$, the boundedness of $\|u(t)\|_{L_{x}^{\infty}}$ is obviously not guaranteed even if we have $\|\theta(t)\|_{L_{x}^{\infty}}\leq \|\theta_{0}\|_{L_{x}^{\infty}}$. To overcome this kind of difficulty, one may resort to following logarithmic Sobolev interpolation inequality
\begin{equation}\label{cdefw34tg}\|f\|_{L^{\infty}}\leq C(1+\|f\|_{L^{2}}+\|f\|_{\dot{B}_{\infty,\,\infty}^{0}}\ln \Big(e+\|\Lambda^{\sigma}f\|_{L^{2}} \Big),\quad \forall \sigma>1. \end{equation}
Invoking several techniques and (\ref{cdefw34tg}), the resulting corresponding $H^1$-estimate of $\theta$ is of the following differential inequality with some $\varrho>1$
\begin{equation}\label{jhhk98078}
\frac{d}{dt}A(t)+
 B(t)
 \leq \widetilde{C_{1}} \big(A(t)+e\big)+\widetilde{C_{2}} \Big(\ln\big(A(t)+B(t)+e \big)\Big)^{\varrho}\big(A(t)+e\big)
 \end{equation}
for some absolute constants $\widetilde{C_{1}}>0$ and $\widetilde{C_{2}}>0$.
With \eqref{jhhk98078} in hand, the natural next step would be to make use of the logarithmic Gronwall
inequality, but the power $\varrho>1$ leads to the unavailability of the known Gronwall inequality, also including the very recent result (Lemma 2.3 of \cite{LT}). This motives us to consider the relationship between $A(t)$ and $B(t)$.
As a matter of fact, by fully exploiting of the dissipation of the SQG equation (\ref{SQG}), we obtain the key estimate
\begin{equation}\label{sdjhhk98074}B(t)\geq C_{1}A^{\gamma}(t),\quad \gamma>1\end{equation}
for some absolute constant $C_{1}>0$.
Fortunately, if the relationship $(\ref{sdjhhk98074})$ holds, then it indeed implies the boundedness of the quantity $A(t)$ (see Lemma \ref{addLdem01} for details), which is nothing but the desired global $H^1$-estimate.
Next, we are able to obtain the global $H^2$-estimate by combining the anisotropic Sobolev inequality (see Lemma \ref{triple}) and the obtained global $H^1$-estimate. Finally,
the global existence of $H^{s}$-estimate follows directly.

\vskip .2in

The method adopted in proving Theorem \ref{ThSQG} may also be adapted with almost no change to the study of a more general case: $u=\mathbf{T}[\theta]$, where $\mathbf{T}$ is a divergence
free zero order operator. For example, we consider the following 2D incompressible porous
medium equation with partial dissipation:
\begin{equation}\label{PPM}
\left\{\aligned
&\partial_{t}\theta+( {u}\cdot\nabla)\theta+\mu\Lambda_{x_{1}}^{2\alpha}\theta+\nu\Lambda_{x_{2}}^{2\beta}\theta=0, \quad x=(x_{1},x_{2})\in \mathbb{R}^2, \,\, t>0, \\
& {u}=-\nabla p-\theta e_{2},\\
&\nabla\cdot {u}=0,\\
&\theta(x, 0)=\theta_{0}(x).
\endaligned\right.
\end{equation}
More precisely, the result can be stated as follows.
\begin{thm}\label{ThSpmx}
Let $\theta_{0}\in H^{s}(\mathbb{R}^{2})$ for $s\geq2$. If $\alpha\in (0,\,1)$ and $\beta\in (0,\,1)$ satisfy (\ref{sdf2334}),
then the system (\ref{PPM})
admits a unique global solution $\theta$ such that for any given $T>0$,
$$\theta\in L^{\infty}([0, T]; H^{s}(\mathbb{R}^{2})),\quad \Lambda_{x_{1}}^{\alpha}\theta,\ \Lambda_{x_{2}}^{\beta}\theta \in L^{2}([0, T]; H^{s}(\mathbb{R}^{2})).
$$
\end{thm}
\begin{rem}
As a matter of fact, the equation ${u}=-\nabla p-\theta e_{2}$ and the incompressible condition $\nabla\cdot {u}=0$ allow us to conclude
$$ {u}=(-\mathcal{R}_{1}\mathcal{R}_{2}\theta,\,\,\mathcal
{R}_{1}\mathcal {R}_{1}\theta).$$
Whence, performing the same manner as adopted in proving Theorem \ref{ThSQG}, one may complete the proof of Theorem \ref{ThSpmx} immediately. To avoid redundancy, we omit the details.
 \end{rem}

\vskip .2in
The present paper is organized as follows. In Section 2, we provide several useful lemmas which play a key role in the main proof. Then we dedicate to the proof of Theorem \ref{ThSQG} in Section 3. Besov spaces and several inequalities are collected in Appendix \ref{appd11}. For convenience, we present the local well-posedness theory of (\ref{SQG}) in Appendix \ref{appd12}.

 \vskip .4in
\section{Preliminaries}\setcounter{equation}{0}
In this section, we collect some preliminary results, including a logarithmic type
Gronwall inequality, an anisotropic Sobolev inequality and several interpolation inequalities involving fractional derivatives, which will be used in the rest
of this paper. In this paper, all constants will be denoted
by $C$ that is a generic constant depending only on the quantities specified in the context.
If we need $C$ to depend on a parameter, we shall indicate this by subscripts.

\vskip .1in
We first establish the following logarithmic type Gronwall
inequality which will play an important role in the proof of Theorem
\ref{ThSQG}.
\begin{lemma}\label{addLdem01}
Assume that $l(t),\,m(t),\,n(t)$ and $f(t)$ are all nonnegative and integrable functions on $(0, T)$ for any given $T>0$.
Let $A\geq 0$ and $B\geq 0$ be two
absolutely continuous functions on $(0, T)$ satisfying for any $t\in (0, T)$
\begin{equation} \label{difeqedd}
A'(t)+B(t)\leq \Big[l(t)+m(t)\ln \big(A(t)+e\big) +n(t)\big(\ln(A(t)+B(t)+e)\big)^{\alpha}\Big] \big(A(t)+e\big)+f(t)\quad \end{equation}
with $\alpha>1$.
Assume further that for some positive constant $C_{1}>0$
\begin{equation} \label{ffscx1eef}B(t)\geq C_{1}A^{\gamma}(t),\quad \gamma>1,\end{equation}
and for constants $K\in[0,\,\infty)$, $\beta\in [0,\,\frac{\gamma-1}{\gamma})$ such that for any $t\in (0, T)$
\begin{equation}
n(t)\leq K \big(A(t)+B(t)+e\big)^{\beta}.\nonumber
\end{equation}
Then the following estimate holds true
\begin{equation} \label{ffscx}A(t)+\int_{0}^{t}{B(s)\,ds}\leq
\widetilde{C}(C_{1},l,m,n,f,\alpha,\beta,\gamma,K,t),\end{equation}
for any $t\in (0, T)$. In particular, for the case $\beta=0$, namely,
\begin{equation}
n(t)\leq K,\nonumber
\end{equation}
the estimate (\ref{ffscx}) still holds true.
\end{lemma}
\begin{rem}\rm
It is worthwhile to mention that Li-Titi \cite{LT} established a logarithmic type Gronwall
inequality with $\alpha\leq1$ and $\beta=0$, but without the restriction (\ref{ffscx1eef}), we also refer to Cao-Li-Titi \cite{CLT} for more general result. We also point out that the restriction $\alpha\leq1$ is a crucial condition in the previous works. In fact, the differential inequality (\ref{difeqedd}) with $\alpha>1$ appears easily when we handle the well-posedness issue of PDEs. By take fully exploit of the hidden information of the fluid mechanic with some certain dissipation, we have the key observation that the condition (\ref{ffscx1eef}) may be true, and thus it can relax $\alpha$ to $\alpha>1$. This motives us to establish a logarithmic type Gronwall
inequality like Lemma \ref{addLdem01}, which may be of independent interest and potential applications.
\end{rem}

\begin{proof}[Proof of Lemma \ref{addLdem01}]
First, denoting
$$A_{1}:= A+e+\sigma,\qquad B_{1}:= A+B+e+\sigma,$$
where $\sigma>0$ to be fixed hereafter, we thus obtain
\begin{align}\label{G001}
A'_{1}+B_{1} =&A'+A+B+e+\sigma\nonumber\\
 \leq&\Big[l(t)+m(t)\ln(A+e)+n(t)\big(\ln(A+B+e)\big)^{\alpha}\Big](A+e)+A+e+\sigma+f(t)
\nonumber\\
 =&\Big[l(t)+m(t)\ln (A_{1}-\sigma)+n(t)\big(\ln (B_{1}-\sigma)\big)^{\alpha}\Big](A_{1}-\sigma)+A_{1}+f(t)
\nonumber\\
 \leq&\Big[1+l(t)+m(t)\ln A_{1}+n(t)\big(\ln B_{1}\big)^{\alpha}\Big]A_{1}+f(t).
\end{align}
Dividing both sides of the above differential inequality (\ref{G001}) by $A_{1}$ and using the fact $A_{1}\geq1$, we further have
\begin{eqnarray}\label{G002}(\ln A_{1})'+\frac{B_{1}}{A_{1}}\leq 1+l(t)+m(t)\ln A_{1}+n(t)\big(\ln B_{1}\big)^{\alpha} + f(t).\end{eqnarray}
It follows from (\ref{ffscx1eef}) that
\begin{equation} \label{322gye001}
B_{1}(t)\geq \frac{C_{1}}{2^{\gamma-1}}A_{1}^{\gamma}(t),\quad \gamma>1 .
\end{equation}
As a matter of fact, one has
\begin{align}
B_{1} =&A+B+e+\sigma 
 = A_{1}+B  
\geq A_{1}+C_{1}A^{\gamma}
= A_{1}+C_{1}(A_{1}-e-\sigma)^{\gamma}
\nonumber\\
=&\Big(\frac{1}{A_{1}^{\gamma-1}}+C_{1}\big(1-\frac{e+\sigma}{A_{1}}\big)^{\gamma}\Big) A_{1}^{\gamma}
\geq  \max\Big\{ \frac{1}{2^{\gamma-1}(\sigma+e)^{\gamma-1}},\,\frac{C_{1}}{2^{\gamma}}\Big\}A_{1}^{\gamma}
\nonumber\\
\geq&  \frac{C_{1}}{2^{\gamma-1}} A_{1}^{\gamma},\nonumber
\end{align}
where in the sixth line we have used
\begin{equation*}
\frac{1}{A_{1}^{\gamma-1}}+C_{1}\big(1-\frac{\sigma+e}{A_{1}}\big)^{\gamma}\geq
\left\{\aligned
& \frac{1}{2^{\gamma-1}(\sigma+e)^{\gamma-1}},  \qquad \sigma+e\leq A_{1}\leq 2(\sigma+e),\\
&\frac{C_{1}}{2^{\gamma}},   \quad  \qquad \qquad \ \qquad A_{1}\geq2(\sigma+e),
\endaligned\right.
\end{equation*}
and in the last line we have taken $\sigma$ satisfying
$$\sigma\geq \Big(\frac{2}{C_{1}}\Big)^{\frac{1}{\gamma-1}}-e\Rightarrow \frac{C_{1}}{2^{\gamma}}\geq\frac{1}{2^{\gamma-1}(\sigma+e)^{\gamma-1}}.$$
Now under the assumption of (\ref{322gye001}), we will show the key bound
\begin{equation} \label{322gye002}
\big(\ln B_{1}\big)^{\alpha}\leq C_{2}\frac{B_{1}^{\theta_{1}}}{A_{1}^{\theta_{2}}}+C_{3}\ln A_{1},
\end{equation}
where $C_{3},\,C_{4},\,\theta_{1},\,\theta_{2}$ are positive constants satisfying $\theta_{2}< \gamma \theta_{1}$. To this end, we define a function
$$F(B_{1})=C_{2}\frac{B_{1}^{\theta_{1}}}{A_{1}^{\theta_{2}}}+C_{3}\ln A_{1}-\big(\ln B_{1}\big)^{\alpha}.$$
Next we will find some conditions to guarantee that $F(B_{1})$ is a nondecreasing function for $B_{1} \geq \frac{C_{1}}{2^{\gamma-1}}A_{1}^{\gamma}$. As a result, if (\ref{322gye002}) holds, then it suffices
$$F(B_{1})\geq F\big(\frac{C_{1}}{2^{\gamma-1}}A_{1}^{\gamma}\big)=\frac{C_{2}C_{1}^{\theta_{1}}}{2^{(\gamma-1)\theta_{1}}}A_{1}^{\gamma \theta_{1}-\theta_{2}}+C_{3}\ln A_{1}-\Big(\ln C_{1}-(\gamma-1)\ln 2+\gamma \ln A_{1}\Big)^{\alpha}.$$
Thanks to $\theta_{2}< \gamma \theta_{1}$, it is not hard to check that there exists a suitable large $\sigma_{1}=\sigma_{1}(C_{1},C_{2},C_{3},\alpha,\gamma,\theta_{1},\theta_{2})>0$ such that for all $\sigma\geq \sigma_{1}$, we have
$$\frac{C_{2}C_{1}^{\theta_{1}}}{2^{(\gamma-1)\theta_{1}}}A_{1}^{\gamma \theta_{1}-\theta_{2}}+C_{3}\ln A_{1}-\Big(\ln C_{1}-(\gamma-1)\ln 2+\gamma \ln A_{1}\Big)^{\alpha}\geq0.$$
In order to show the non decreasing property of $F(B_{1})$, we differentiate it to get
$$F'(B_{1})=\Big(C_{2}\theta_{1}\frac{B_{1}^{\theta_{1}}}{A_{1}^{\theta_{2}}}-\alpha\big(\ln B_{1}\big)^{\alpha-1}\Big)\frac{1}{B_{1}}.$$
By the fact $B_{1} \geq \frac{C_{1}}{2^{\gamma-1}}A_{1}^{\gamma}$, one has
$$C_{2}\theta_{1}\frac{B_{1}^{\theta_{1}}}{A_{1}^{\theta_{2}}}-\alpha\big(\ln B_{1}\big)^{\alpha-1}\geq
C_{2}\theta_{1}\Big(\frac{C_{1}}{2^{\gamma-1}}\Big)^{\frac{\theta_{2}}{\gamma}}
B_{1}^{\theta_{1}-\frac{\theta_{2}}{\gamma}}-\alpha\big(\ln B_{1}\big)^{\alpha-1}.$$
Similarly, one can show that there exists a suitable large $\sigma_{2}=\sigma_{2}(C_{1},C_{2},\alpha,\gamma,\theta_{1},\theta_{2})>0$ such that for all $\sigma\geq \sigma_{2}$, we obtain
$$C_{2}\theta_{1}\Big(\frac{C_{1}}{2^{\gamma-1}}\Big)^{\frac{\theta_{2}}{\gamma}}
B_{1}^{\theta_{1}-\frac{\theta_{2}}{\gamma}}-\alpha\big(\ln B_{1}\big)^{\alpha-1}\geq0.$$
Now the above bound yields that $F'(B_{1})\geq0$ for $B_{1} \geq \frac{C_{1}}{2^{\gamma-1}}A_{1}^{\gamma}$.
Combining the above analysis, if we take $\sigma\geq \max\{\sigma_{1},\,\sigma_{2}\}$, then the desired (\ref{322gye002}) indeed holds.
Notice that
$$ n(t)\leq K \big(A(t)+B(t)+e\big)^{\beta}\leq KB_{1}^{\beta}.$$
and using (\ref{322gye002}), it is not hard to check
\begin{align}\label{G003}
n(t)\big(\ln B_{1}\big)^{\alpha} \leq&n(t)\Big(C_{2}\frac{B_{1}^{\theta_{1}}}{A_{1}^{\theta_{2}}}+C_{3}\ln A_{1}\Big) 
 = 
C_{2}n(t)\frac{B_{1}^{\theta_{1}}}{A_{1}^{\theta_{2}}}+C_{3}n(t)\ln A_{1}
\nonumber\\
 \leq&
C_{2}KB_{1}^{\beta}\frac{B_{1}^{\theta_{1}}}{A_{1}^{\theta_{2}}}+C_{3}n(t)\ln A_{1}
 = 
C_{2}K \Big( \frac{B_{1}}{A_{1}}\Big)^{\beta+\theta_{1}} +C_{3}n(t)\ln A_{1}
\nonumber\\
 \leq&
\frac{B_{1}}{2A_{1}}+C(C_{2}, \theta_{1},\theta_{2},\alpha,\beta,K)+C_{3}n(t)\ln A_{1},
\end{align}
where we have used the following condition
$$\theta_{2}=\beta+\theta_{1}<1.$$
This along with $\theta_{2}< \gamma \theta_{1}$ implies
$$\frac{\beta}{\gamma-1}<\theta_{1}<1-\beta,$$
which leads to the restriction
$$\beta<\frac{\gamma-1}{\gamma}.$$
Therefore, we first fix $C_{2},\,C_{3}$, $\theta_{1}$ and $\theta_{2}$, then we choose
$$\sigma\geq \max\left\{\Big(\frac{2}{C_{1}}\Big)^{\frac{1}{\gamma-1}}-e,\ \sigma_{1},\ \sigma_{2} \right\},$$
where $\sigma_{1}=\sigma_{1}(C_{1},\alpha,\beta,\gamma)$ and $\sigma_{2}=\sigma_{2}(C_{1},\alpha,\beta,\gamma)>0$.
Summing up (\ref{G002}) and (\ref{G003}), we conclude
\begin{eqnarray}\label{G004}(\ln A_{1})'+\frac{B_{1}}{2A_{1}}\leq\big(m(t)+C_{3}n(t)\big)\ln A_{1}+C(C_{1},\alpha,\beta,\gamma)+l(t)+ f(t).\end{eqnarray}
For the sake of simplicity, we denote
$$X(t):=\ln A_{1}(t)+\int_{0}^{t}{\frac{B_{1}(s)}{2A_{1}(s)}\,ds},$$
then it follows from (\ref{G004}) that
$$X'(t)\leq
C(C_{1},\alpha,\beta,\gamma) +l(t)+ f(t)+\big(m(t)+C_{3}n(t)\big)X(t).$$
Whereas by using a standard Gronwall inequality, we obtain
\begin{align}
X(t) \leq& e^{\int_{0}^{t}{\big(m(s)+ C_{3} n(s)\big)\,ds}}\Big(X(0)+\int_{0}^{t}{\{(C_{1},\alpha,\beta,\gamma)
+l(s)+f(s)\}\,ds}\Big)\nonumber\\
:=&C(C_{1},l,m,n,f,\alpha,\beta,\gamma,K,t).\nonumber
\end{align}
According to the definition of $X$, we infer
$$A_{1}(t)\leq e^{X(t)}\leq e^{C(C_{1},l,m,n,f,\alpha,\beta,\gamma,K,t)}.$$
Moreover, it is also easy to see that
\begin{align}
\int_{0}^{t}{ B_{1}(s) \,ds} =&\int_{0}^{t}{2A_{1}(s)\frac{B_{1}(s)}{2A_{1}(s)}\,ds}
 \leq
\int_{0}^{t}{2\big(\max_{0\leq \tau\leq t}A_{1}(\tau)\big)\frac{B_{1}(s)}{2A_{1}(s)}\,ds}\nonumber\\
\leq& 2e^{C(C_{1},l,m,n,f,\alpha,\beta,\gamma,K,t)}\int_{0}^{t}{ \frac{B_{1}(s)}{2A_{1}(s)}\,ds}\nonumber\\
\leq& 2C(C_{1},l,m,n,f,\alpha,\beta,\gamma,K,t)
e^{C(C_{1},l,m,n,f,\alpha,\beta,\gamma,K,t)}.\nonumber
\end{align}
This concludes the proof of Lemma \ref{addLdem01}.
\end{proof}

\vskip .1in
The following anisotropic
Sobolev inequalities
will be frequently used later.
\begin{lemma}\label{gfhgj}
The following anisotropic interpolation inequalities hold true for $i=1,\,2$
\begin{eqnarray}\label{t2326001}
 \|\Lambda_{x_{i}}^{s}f\|_{L^{2}}\leq
C\|f\|_{L^{2}}^{1-\frac{s}{\delta+1}}\|\Lambda_{x_{i}}^{\delta}\partial_{x_{i}}f\|_{L^{2}}^{\frac{s}{\delta+1}},
\end{eqnarray}
where $0\leq s\leq \delta+1$. In particular, we have
\begin{eqnarray}\label{adt2326001}
\|\Lambda_{x_{i}}^{\gamma}f\|_{L^{2}}\leq C\|f\|_{L^{2}}^{1-\frac{\gamma}{\varrho}}\|\Lambda_{x_{i}}
^{\varrho}f\|_{L^{2}}^{\frac{\gamma}{\varrho}},\quad  0\leq\gamma\leq\varrho.
_{}\end{eqnarray}
\end{lemma}
\begin{proof}[ {Proof of Lemma \ref{gfhgj}}]
It suffices to show (\ref{t2326001}) for $i=1$ as $i=2$ can be performed as the same manner. By the interpolation inequality and the Young inequality, it is obvious to check that
\begin{align}
\|\Lambda_{x_{1}}^{s}f\|_{L^{2}}^{2}
=&\int_{\mathbb{R}}\int_{\mathbb{R}}{|\Lambda_{x_{1}}^{s}f(x_{1},\,x_{2})|^{2}\,
dx_{1}dx_{2}}\nonumber\\
=&\int_{\mathbb{R}}{ \|\Lambda_{x_{1}}^{s}f(x_{1},\,x_{2})\|_{L_{x_{1}}^{2}}^{2}\,
 dx_{2}}\nonumber\\
 \leq& C\int_{\mathbb{R}}{ \|f(x_{1},\,x_{2})\|_{L_{x_{1}}^{2}}^{2(1-\frac{s}{\delta+1})}
 \|\Lambda_{x_{1}}^{\delta}\partial_{x_{1}}f(x_{1},\,x_{2})\|_{L_{x_{1}}^{2}}^{\frac{2s}{\delta+1}}\,
 dx_{2}}
 \nonumber\\
 \leq& C\left(\int_{\mathbb{R}}{ \|f(x_{1},\,x_{2})\|_{L_{x_{1}}^{2}}^{2}\,
 dx_{2}}\right)^{1-\frac{s}{\delta+1}}
 \left(\int_{\mathbb{R}}{
 \|\Lambda_{x_{1}}^{\delta}\partial_{x_{1}}f(x_{1},\,x_{2})\|_{L_{x_{1}}^{2}}^{2}\,
 dx_{2}}\right)^{\frac{s}{\delta+1}}
 \nonumber\\
 =&
 C\|f\|_{L^{2}}^{2(1-\frac{s}{\delta+1})}
 \|\Lambda_{x_{1}}^{\delta}\partial_{x_{1}}f\|_{L^{2}}^{\frac{2s}{\delta+1}}, \nonumber
\end{align}
which is nothing but the desired result (\ref{t2326001}). Following the proof of (\ref{t2326001}), the estimate (\ref{adt2326001}) immediately holds true.
This completes the proof of the lemma.
\end{proof}

\vskip .1in
We also need the following anisotropic
Sobolev inequalities.
\begin{lemma} \label{dfdasfwqew33}
The following anisotropic interpolation inequalities hold true for $i=1,\,2$
\begin{eqnarray}\label{ghpoiu908}
 \|\partial_{x_{i}}f\|_{L^{2(\gamma+1)}}\leq
C\|f\|_{L^{\infty}}^{\frac{\gamma}{\gamma+1}}\|\Lambda_{x_{i}}^{\gamma}\partial_{x_{i}}f\|_{L^{2}}^{\frac{1}{\gamma+1}}
,\quad \gamma\geq0,
\end{eqnarray}
\begin{eqnarray}\label{ghpoiu9ddfgh}
 \|\Lambda_{x_{i}}^{\delta}f\|_{L^{\frac{2(\varrho+1)}{\delta}}}\leq
C\|f\|_{L^{\infty}}^{1-\frac{\delta}{\varrho+1}}\|\Lambda_{x_{i}}^{\varrho}
\partial_{x_{i}}f\|_{L^{2}}^{\frac{\delta}{\varrho+1}},\quad 0\leq\delta\leq\varrho+1.
\end{eqnarray}
\end{lemma}
\begin{proof}[ {Proof of Lemma \ref{dfdasfwqew33}}]
It is sufficient to prove (\ref{ghpoiu908}) and (\ref{ghpoiu9ddfgh}) for $i=1$.
We first recall the following one-dimensional Sobolev inequality
\begin{eqnarray}
\|\partial_{x_{1}}g\|_{L_{x_{1}}^{2(\gamma+1)}(\mathbb{R})}\leq C\|g\|_{L_{x_{1}}^{\infty}(\mathbb{R})}^{\frac{\gamma}{\gamma+1}} \|\Lambda_{x_{1}}^{\gamma}\partial_{x_{1}}g\|_{L_{x_{1}}^{2}(\mathbb{R})}^{ \frac{1}{\gamma+1}},\nonumber
\end{eqnarray}
where we have used the sub-index $x_{1}$ with the Lebesgue spaces to emphasize that the norms are taken in one-dimensional Lebesgue spaces with respect to $x_{1}$.
Thanks to the above interpolation inequality and the Young inequality, we have
\begin{align}
\|\partial_{x_{1}}f\|_{L^{2(\gamma+1)}}^{2(\gamma+1)} =&\int_{\mathbb{R}}{
\|\partial_{x_{1}}f(x_{1},\,x_{2})\|_{L^{2(\gamma+1)}}^{2(\gamma+1)}\,
dx_{2}}\nonumber\\ \leq& C\int_{\mathbb{R}}{ \|f(x_{1},\,x_{2})\|_{L_{x_{1}}^{\infty}}^{2\gamma}
 \|\Lambda_{x_{1}}^{\gamma}\partial_{x_{1}}f(x_{1},\,x_{2})\|_{L_{x_{1}}^{2}}^{2}\,
 dx_{2}}
 \nonumber\\ \leq& C\|f(x_{1},\,x_{2})\|_{L_{x_{1}x_{2}}^{\infty}}^{2\gamma}\int_{\mathbb{R}}{
 \|\Lambda_{x_{1}}^{\gamma}\partial_{x_{1}}f(x_{1},\,x_{2})\|_{L_{x_{1}}^{2}}^{2}\,
 dx_{2}}
 \nonumber\\ =&
 C\|f\|_{L^{\infty}}^{2\gamma}
 \|\Lambda_{x_{1}}^{\gamma}\partial_{x_{1}}f\|_{L^{2}}^{2}, \nonumber
\end{align}
which implies that
$$\|\partial_{x_{1}}f\|_{L^{2(\gamma+1)}}\leq
C\|f\|_{L^{\infty}}^{\frac{\gamma}{\gamma+1}}
\|\Lambda_{x_{1}}^{\gamma}\partial_{x_{1}}f\|_{L^{2}}^{\frac{1}{\gamma+1}}.$$
Similarly, using the following one-dimensional Sobolev inequality
$$\|\Lambda_{x_{1}}^{\delta}g\|_{L^{\frac{2(\varrho+1)}{\delta}}}\leq C
\|g\|_{L_{x_{1}}^{\infty}}^{\frac{\varrho+1-\delta}{\varrho+1}}
 \|\Lambda_{x_{1}}^{\varrho}\partial_{x_{1}}g\|_{L_{x_{1}}^{2}}^{\frac{\delta}{\varrho+1}},$$
one may conclude
\begin{align}
\|\Lambda_{x_{1}}^{\delta}f\|_{L^{\frac{2(\varrho+1)}{\delta}}}^{\frac{2(\varrho+1)}
{\delta}} =&\int_{\mathbb{R}}{
\|\Lambda_{x_{1}}^{\delta}f(x_{1},\,x_{2})\|_{L^{\frac{2(\varrho+1)}
{\delta}}}^{\frac{2(\varrho+1)}{\delta}}\,
dx_{2}}\nonumber\\
\leq& C\int_{\mathbb{R}}{ \|f(x_{1},\,x_{2})\|_{L_{x_{1}}^{\infty}}^{\frac{2(\varrho+1-\delta)}{\delta}}
 \|\Lambda_{x_{1}}^{\varrho}\partial_{x_{1}}f(x_{1},\,x_{2})\|_{L_{x_{1}}^{2}}^{2}\,
 dx_{2}}
 \nonumber\\ \leq& C\|f(x_{1},\,x_{2})\|_{L_{x_{1}x_{2}}^{\infty}}^{\frac{2(\varrho+1-\delta)}{\delta}}\int_{\mathbb{R}}{
 \|\Lambda_{x_{1}}^{\varrho}\partial_{x_{1}}f(x_{1},\,x_{2})\|_{L_{x_{1}}^{2}}^{2}\,
 dx_{2}}
 \nonumber\\ =&
 C\|f\|_{L^{\infty}}^{\frac{2(\varrho+1-\delta)}{\delta}}
 \|\Lambda_{x_{1}}^{\varrho}\partial_{x_{1}}f\|_{L^{2}}^{2},\nonumber
\end{align}
which leads to the following desired estimate
$$ \|\Lambda_{x_{1}}^{\delta}f\|_{L^{\frac{2(\varrho+1)}{\delta}}}\leq
C\|f\|_{L^{\infty}}^{1-\frac{\delta}{\varrho+1}}\|\Lambda_{x_{1}}^{\varrho}
\partial_{x_{1}}f\|_{L^{2}}^{\frac{\delta}{\varrho+1}}.$$
We therefore conclude the proof of Lemma \ref{dfdasfwqew33}.
\end{proof}

\vskip .1in

In order to obtain the higher regularity, we need to establish the following anisotropic Sobolev inequality.
\begin{lemma} \label{triple}
Let $f\in L_{x_{2}}^{q}L_{x_{1}}^{p}(\mathbb{R}^2)$ for $p,\,q\in[2,\,\infty]$. If $g, \ h\in L^2(\mathbb{R}^2)$, $\Lambda_{x_{1}}^{\gamma_{1}}g,\ \Lambda_{x_{2}}^{\gamma_{2}}h \in L^2(\mathbb{R}^2)$ for any $\gamma_{1}\in (\frac{1}{p},\,1]$ and $\gamma_{2}\in (\frac{1}{q},\,1]$, then it holds true
\begin{equation}
 \int_{\mathbb{R}}\int_{\mathbb{R}}| f \, g\, h|  \;dx_{1}dx_{2}  \le C \, \|f\|_{L_{x_{2}}^{q}L_{x_{1}}^{p}} \, \|g\|_{L^2}^{1-\frac{1}{\gamma_{1}p}}\|\Lambda_{x_{1}}^{\gamma_{1}}g\|_{L^2}^{ \frac{1}{\gamma_{1}p}} \,\|h\|_{L^2}^{1-\frac{1}{\gamma_{2}q}}\|\Lambda_{x_{2}}^{\gamma_{2}}h\|_{L^2}^{ \frac{1}{\gamma_{2}q}},\nonumber
\end{equation}
where here and in sequel, we use the notation
$$\|h\|_{L_{x_{2}}^{q}L_{x_{1}}^p}:=\Big( \int_{\mathbb{R}}\|h(.,x_{2})\|_{L_{x_{1}}^p}^{q} \,dx_{2}\Big)^{\frac{1}{q}}.$$
In particular, let $f,\ g, \ h\in L^2(\mathbb{R}^2)$ and $\Lambda_{x_{1}}^{\gamma_{1}}g,\ \Lambda_{x_{2}}^{\gamma_{2}}h \in L^2(\mathbb{R}^2)$ for any $\gamma_{1},\ \gamma_{2}\in (\frac{1}{2},\,1]$, then it holds true
\begin{equation} \label{qtri}
 \int_{\mathbb{R}}\int_{\mathbb{R}}| f \, g\, h|  \;dx_{1}dx_{2}  \le C \, \|f\|_{L^2} \, \|g\|_{L^2}^{1-\frac{1}{2\gamma_{1}}}\|\Lambda_{x_{1}}^{\gamma_{1}}g\|_{L^2}^{ \frac{1}{2\gamma_{1}}} \,\|h\|_{L^2}^{1-\frac{1}{2\gamma_{2}}}\|\Lambda_{x_{2}}^{\gamma_{2}}h\|_{L^2}^{ \frac{1}{2\gamma_{2}}},
\end{equation}
where $C$ is a constant depending on $\gamma_{1}$ and $\gamma_{2}$ only.
\end{lemma}
\begin{proof}[{Proof of Lemma \ref{triple}}]
The proof of this lemma can be found in \cite{WXY17}. For the convenience of the reader, we provide the details.
Now we recall the one-dimensional Sobolev inequality
\begin{eqnarray}\label{t304}
\|g\|_{L_{x_{1}}^{\frac{2p}{p-2}}(\mathbb{R})}\leq C\|g\|_{L_{x_{1}}^{2}(\mathbb{R})}^{1-\frac{1}{\gamma_{1}p}} \|\Lambda_{x_{1}}^{\gamma_{1}}g\|_{L_{x_{1}}^{2}(\mathbb{R})}^{ \frac{1}{\gamma_{1}p}},\qquad \gamma_{1}\in \Big(\frac{1}{p},\,1\Big],
\end{eqnarray}
where here and in what follows, we adopt the convention $\frac{2p}{p-2}=\infty$ for $p=2$.
By means of (\ref{t304}) and the H$\rm\ddot{o}$lder inequality, one deduces
\begin{align}\label{tr001}
  \int_{\mathbb{R}}\int_{\mathbb{R}}| f \, g\, h|  \;dx_{1}dx_{2} \leq& C \, \int_{\mathbb{R}}\|f\|_{L_{x_{1}}^p} \, \|g\|_{L_{x_{1}}^{\frac{2p}{p-2}}}\,\|h\|_{L_{x_{1}}^2} \,dx_{2}\nonumber\\ \leq&C \, \int_{\mathbb{R}}\|f\|_{L_{x_{1}}^p} \, \|g\|_{L_{x_{1}}^2}^{1-\frac{1}{\gamma_{1}p}}\|\Lambda_{x_{1}}^{\gamma_{1}}g\|_{L_{x_{1}}^2}^{ \frac{1}{\gamma_{1}p}} \,\|h\|_{L_{x_{1}}^2} \,dx_{2}\nonumber\\ \leq& C\Big( \int_{\mathbb{R}}\|f\|_{L_{x_{1}}^p}^{q} \,dx_{2}\Big)^{\frac{1}{q}} \Big( \int_{\mathbb{R}}\|g\|_{L_{x_{1}}^2}^{2} \,dx_{2}\Big)^{\frac{\gamma_{1}p-1}{2\gamma_{1}p}}\nonumber\\& \times
 \Big( \int_{\mathbb{R}}\|\Lambda_{x_{1}}^{\gamma_{1}}g\|_{L_{x_{1}}^2}^{2} \,dx_{2}\Big)^{\frac{1}{2\gamma_{1}p}}\|h\|_{L_{x_{2}}^{\frac{2q}{q-2}}L_{x_{1}}^2}
 \nonumber\\ =& C\|f\|_{L_{x_{2}}^{q}L_{x_{1}}^{p}} \, \|g\|_{L^2}^{1-\frac{1}{\gamma_{1}p}}\|\Lambda_{x_{1}}^{\gamma_{1}}g\|_{L^2}^{ \frac{1}{\gamma_{1}p}}  \,\|h\|_{L_{x_{2}}^{\frac{2q}{q-2}}L_{x_{1}}^2}.
\end{align}
According to the Minkowski inequality and (\ref{t304}), we have
\begin{align}\label{tr002}
\|h\|_{L_{x_{2}}^{\frac{2q}{q-2}}L_{x_{1}}^2} \leq& C\Big(\int_{\mathbb{R}}\|h(x_{1},\,x_{2})\|_{L_{x_{2}}^{\frac{2q}{q-2}}}^{2}
\,dx_{1}\Big)^{\frac{1}{2}}
\nonumber\\ \leq& C\Big(\int_{\mathbb{R}}\|h(x_{1},\,x_{2})\|_{L_{x_{2}}^2}^{2-\frac{2}{\gamma_{2}q}}
\|\Lambda_{x_{2}}^{\gamma_{2}}h(x_{1},\,x_{2})\|_{L_{x_{2}}^2}^{ \frac{2}{\gamma_{2}q}}\,dx_{1}\Big)^{\frac{1}{2}}\nonumber\\ \leq& C\Big(\int_{\mathbb{R}}\|h(x_{1},\,x_{2})\|_{L_{x_{2}}^2}^{2}\,dx_{1}\Big)
^{\frac{\gamma_{2}q-1}{2\gamma_{2}q}} \Big(\int_{\mathbb{R}}
\|\Lambda_{x_{2}}^{\gamma_{2}}h(x_{1},\,x_{2})\|_{L_{x_{2}}^2}^{2}\,dx_{1}
\Big)^{\frac{1}
{2\gamma_{2}q}}
\nonumber\\ =&C\|h\|_{L^2}^{1-\frac{1}{\gamma_{2}q}}\|\Lambda_{x_{2}}^{\gamma_{2}}h\|_{L^2}^{ \frac{1}{\gamma_{2}q}}.
\end{align}
Inserting (\ref{tr002}) into (\ref{tr001}) gives
$$\int_{\mathbb{R}}\int_{\mathbb{R}}| f \, g\, h|  \;dx_{1}dx_{2}  \le C \, \|f\|_{L_{x_{2}}^{q}L_{x_{1}}^{p}} \, \|g\|_{L^2}^{1-\frac{1}{\gamma_{1}p}}\|\Lambda_{x_{1}}^{\gamma_{1}}g\|_{L^2}^{ \frac{1}{\gamma_{1}p}} \,\|h\|_{L^2}^{1-\frac{1}{\gamma_{2}q}}\|\Lambda_{x_{2}}^{\gamma_{2}}h\|_{L^2}^{ \frac{1}{\gamma_{2}q}},$$
which is the desired inequality (\ref{qtri}). This completes the proof of the lemma.
\end{proof}

\vskip .1in

Finally, the following standard commutator estimate will also be used as well, which can be found in \cite[p.614]{KPV2}.
\begin{lemma}
Let $s\in(0,1)$ and $p\in (1, \infty)$. Then
\begin{eqnarray}\label{t2326003}
\|\Lambda^s (f\,g) - g\,\Lambda^s f - f\, \Lambda^s g\|_{L^p(\mathbb{R}^d)}
\le C\, \|g\|_{L^\infty(\mathbb{R}^d)}\, \|\Lambda^s f\|_{L^p(\mathbb{R}^d)},
\end{eqnarray}
where $d\ge 1$ denotes the spatial dimension and $C=C(d,s,p)$ is a constant.
In particular, it holds true
$$\|\Lambda^s (f\,g) - f\, \Lambda^s g\|_{L^p(\mathbb{R}^d)}
\le C\, \|g\|_{L^\infty(\mathbb{R}^d)}\, \|\Lambda^s f\|_{L^p(\mathbb{R}^d)}.$$
\end{lemma}

\vskip .2in
\section{The proof of Theorem \ref{ThSQG}}
\setcounter{equation}{0}
The existence and uniqueness of local smooth solutions can be established via a standard procedure (see Appendix \ref{appd12} for details). Thus, in order to complete the proof of Theorem \ref{ThSQG}, it is sufficient to establish {\it a priori} estimates that hold for any fixed $T>0$. The following proposition states the basic bounds.
\begin{Pros}\label{SDFsda2789}
Assume $\theta_{0}$ satisfies the assumptions stated in Theorem \ref{ThSQG} and let $(u, \theta)$ be the corresponding solution. Then, for any $t>0$,
$$\|\theta(t)\|_{L^{2}}^{2}+2\int_{0}^{t}{(\|\Lambda_{x_{1}}^{\alpha}\theta(
\tau)\|_{L^{2}}^{2}+\|\Lambda_{x_{2}}^{\beta}\theta(
\tau)\|_{L^{2}}^{2})\,d\tau}\leq \|\theta_{0}\|_{L^{2}}^{2},$$
$$
\|\theta(t)\|_{L^p}\leq \|\theta_0\|_{L^p}, \qquad 2\le p\le \infty.
$$
\end{Pros}
\begin{proof}[{Proof of Proposition \ref{SDFsda2789}}]
Multiplying the first equation of $(\ref{SQG})$ by $\theta$, using the divergence-free condition and integrating with respect to the space variable, we have
 $$\frac{1}{2}\frac{d}{dt}\|\theta(t)\|_{L^{2}}^{2}+
 \|\Lambda_{x_{1}}^{\alpha}\theta\|_{L^{2}}^{2}+
 \|\Lambda_{x_{2}}^{\beta}\theta\|_{L^{2}}^{2}= 0.$$
Integrating with respect to time yields
\begin{eqnarray}
\|\theta(t)\|_{L^{2}}^{2}+2\int_{0}^{t}{(\|\Lambda_{x_{1}}^{\alpha}\theta(
\tau)\|_{L^{2}}^{2}+\|\Lambda_{x_{2}}^{\beta}\theta(
\tau)\|_{L^{2}}^{2})\,d\tau}\leq \|\theta_{0}\|_{L^{2}}^{2}.\nonumber
\end{eqnarray}
We multiply the first equation of $(\ref{SQG})$ by $|\theta|^{p-2}\theta$ and use the divergence-free condition to derive
$$\frac{1}{p}\frac{d}{dt}\|\theta(t)\|_{L^{p}}^{p}+
\int_{\mathbb{R}^{2}}{\Lambda_{x_{1}}^{2\alpha}\theta(|\theta|^{p-2}\theta)\,dx}+
\int_{\mathbb{R}^{2}}{\Lambda_{x_{2}}^{2\beta}\theta(|\theta|^{p-2}\theta)\,dx}=0.$$
Invoking the lower bounds
\begin{align}
\int_{\mathbb{R}^{2}}{\Lambda_{x_{1}}^{2\alpha}\theta(|\theta|^{p-2}\theta)
\,dx} =&\int_{\mathbb{R}}\int_{\mathbb{R}}{\Lambda_{x_{1}}^{2\alpha}
\theta(x_{1},x_{2})
(|\theta(x_{1},x_{2})|^{p-2}\theta(x_{1},x_{2}))
\,dx_{1}dx_{2}}\nonumber\\ \geq&C\int_{\mathbb{R}}\int_{\mathbb{R}}
{\big(\Lambda_{x_{1}}^{\alpha}
|\theta(x_{1},x_{2})|^{\frac{p}{2}}\big)^{2}
\,dx_{1}dx_{2}}\nonumber
\end{align}
and
\begin{align}
\int_{\mathbb{R}^{2}}{\Lambda_{x_{2}}^{2\beta}\theta(|\theta|^{p-2}\theta)
\,dx} =&\int_{\mathbb{R}}\int_{\mathbb{R}}{\Lambda_{x_{2}}^{2\beta}\theta(x_{1},x_{2})
(|\theta(x_{1},x_{2})|^{p-2}\theta(x_{1},x_{2}))
\,dx_{1}dx_{2}}\nonumber\\ \geq&C\int_{\mathbb{R}}\int_{\mathbb{R}}
{\big(\Lambda_{x_{2}}^{\beta}
|\theta(x_{1},x_{2})|^{\frac{p}{2}}\big)^{2}
\,dx_{1}dx_{2}},\nonumber
\end{align}
it follows that
$$
\|\theta(t)\|_{L^p}\leq \|\theta_0\|_{L^p}, \qquad 2\le p\le \infty.
$$
This ends the proof of the proposition.
\end{proof}
\vskip .1in
We now prove the following global
$H^1$-bound for $\beta>\frac{1}{2\alpha+1}$ and $\beta\geq\alpha$.

\begin{Pros}\label{Lpp302}
Assume $\theta_{0}$ satisfies the assumptions stated in Theorem \ref{ThSQG} and let $(u, \theta)$ be the corresponding solution. If $\alpha$ and $\beta$ satisfy
$$\beta> \frac{1}{2\alpha+1}\ \ \mbox{and}\ \ \beta\geq\alpha,$$
then, for any $t>0$,
\begin{eqnarray}\label{asddfgdfht302}
\|\nabla\theta(t)\|_{L^{2}}^{2}+ \int_{0}^{t}{(\|\Lambda_{x_{1}}^{\alpha}\nabla\theta(
\tau)\|_{L^{2}}^{2}+\|\Lambda_{x_{2}}^{\beta}\nabla\theta(
\tau)\|_{L^{2}}^{2})\,d\tau}\leq
C(t,\,\theta_{0}),
\end{eqnarray}
where $C(t, \,\theta_{0})$ is a constant depending on $t$ and the initial data $\theta_0$.
\end{Pros}

\begin{proof}[{Proof of Proposition \ref{Lpp302}}]
Taking the inner product of $(\ref{SQG})$ with $\Delta \theta$ and using the divergence-free condition $\partial_{x_{1}} u_{1}+\partial_{x_{2}} u_{2}=0$, we infer that
\begin{align}\label{t3326t002}
\frac{1}{2}\frac{d}{dt}\|\nabla \theta(t)\|_{L^{2}}^{2}+
 \|\Lambda_{x_{1}}^{\alpha}\nabla\theta\|_{L^{2}}^{2}+
 \|\Lambda_{x_{2}}^{\beta}\nabla\theta\|_{L^{2}}^{2} =&\int_{\mathbb{R}^{2}}{(u \cdot \nabla)\theta\Delta\theta\,dx}\nonumber\\
 =& \mathcal{H}_{1}+\mathcal{H}_{2}+\mathcal{H}_{3}+\mathcal{H}_{4},
\end{align}
where
$$
\mathcal{H}_{1}=-\int_{\mathbb{R}^{2}}{\partial_{x_{1}}u_{1} \partial_{x_{1}}\theta\partial_{x_{1}}\theta\,dx},\quad \mathcal{H}_{2}=-\int_{\mathbb{R}^{2}}{\partial_{x_{1}}u_{2} \partial_{x_{2}}\theta\partial_{x_{1}}\theta\,dx},
$$
$$
\mathcal{H}_{3}=-\int_{\mathbb{R}^{2}}{\partial_{x_{2}}u_{1} \partial_{x_{1}}\theta\partial_{x_{2}}\theta\,dx},\quad \mathcal{H}_{4}=-\int_{\mathbb{R}^{2}}{\partial_{x_{2}}u_{2} \partial_{x_{2}}\theta\partial_{x_{2}}\theta\,dx}.
$$
In what follows, we shall estimate the terms at the right hand side of (\ref{t3326t002}) one by one.
To estimate the first term, we use $\partial_{x_{1}} u_{1}+\partial_{x_{2}} u_{2}=0$ and the commutator (\ref{t2326003}) to conclude
\begin{align}\label{t3326t003}
\mathcal{H}_{1} =&\int_{\mathbb{R}^{2}}{\partial_{x_{2}}u_{2} \partial_{x_{1}}\theta\partial_{x_{1}}\theta\,dx}\nonumber\\
 =&-2\int_{\mathbb{R}^{2}}{u_{2} \partial_{x_{1}}\theta \partial_{x_{2}x_{1}}\theta\,dx}
\nonumber\\
 =&-2\int_{\mathbb{R}^{2}}{\Lambda_{x_{2}}^{1-\delta}(u_{2} \partial_{x_{1}}\theta) \Lambda_{x_{2}}^{\delta}\Lambda_{x_{2}}^{-1}\partial_{x_{2}}\partial_{x_{1}}\theta\,dx}
\nonumber\\
 \leq&C\|\Lambda_{x_{2}}^{\delta}\partial_{x_{1}}\theta\|_{L^{2}} \|\Lambda_{x_{2}}^{1-\delta}(u_{2} \partial_{x_{1}}\theta)\|_{L^{2}}
\nonumber\\
 \leq&C\|\Lambda_{x_{2}}^{\delta}\partial_{x_{1}}\theta\|_{L^{2}}(\|\Lambda_{x_{2}}^{1-\delta}(u_{2} \partial_{x_{1}}\theta)-\Lambda_{x_{2}}^{1-\delta}u_{2} \partial_{x_{1}}\theta\|_{L^{2}}+\|\Lambda_{x_{2}}^{1-\delta}u_{2} \partial_{x_{1}}\theta\|_{L^{2}})
\nonumber\\
 \leq&C\|\Lambda_{x_{2}}^{\delta}\partial_{x_{1}}\theta\|_{L^{2}}( \|u_{2}\|_{L^{\infty}} \|\Lambda_{x_{2}}^{1-\delta}\partial_{x_{1}}\theta\|_{L^{2}}+ \|\Lambda_{x_{2}}^{1-\delta}u_{2} \partial_{x_{1}}\theta\|_{L^{2}})\nonumber\\
:=&\mathcal{H}_{11}+\mathcal{H}_{12},
\end{align}
where $\mathcal{H}_{11}$ and $\mathcal{H}_{12}$ are given by
$$\mathcal{H}_{11}=C\|\Lambda_{x_{2}}^{\delta}\partial_{x_{1}}\theta\|_{L^{2}} \|u_{2}\|_{L^{\infty}} \|\Lambda_{x_{2}}^{1-\delta}\partial_{x_{1}}\theta\|_{L^{2}}, \ \ \ \ \mathcal{H}_{12}=C\|\Lambda_{x_{2}}^{\delta}\partial_{x_{1}}\theta\|_{L^{2}} \|\Lambda_{x_{2}}^{1-\delta}u_{2} \partial_{x_{1}}\theta\|_{L^{2}}.$$
In light of the interpolation inequality (\ref{adt2326001}), one obtains for $1-\beta\leq\delta<\beta$
\begin{align}
\mathcal{H}_{11} \leq&C\|\partial_{x_{1}}\theta\|_{L^{2}}^{1-\frac{\delta}{\beta}} \|\Lambda_{x_{2}}^{\beta}\partial_{x_{1}}\theta\|_{L^{2}}^{\frac{\delta}{\beta}}  \|u_{2}\|_{L^{\infty}} \|\partial_{x_{1}}\theta\|_{L^{2}}^{1-\frac{1-\delta}{\beta}}
\|\Lambda_{x_{2}}^{\beta}\partial_{x_{1}}\theta\|_{L^{2}}^{\frac{1-\delta}{\beta}}
\nonumber\\
 \leq&C\|\Lambda_{x_{2}}^{\beta}\nabla\theta\|_{L^{2}}^{\frac{1}{\beta}} \|u_{2}\|_{L^{\infty}} \|\nabla\theta\|_{L^{2}}^{2-\frac{1}{\beta}}
\nonumber\\
 \leq&\epsilon\|\Lambda_{x_{2}}^{\beta}\nabla\theta\|_{L^{2}}^{2}+C(\epsilon)
\|u_{2}\|_{L^{\infty}}^{\frac{2\beta}{2\beta-1}} \|\nabla\theta\|_{L^{2}}^{2}.\nonumber
\end{align}
Now we further choose $\delta$ satisfying
$$\frac{1-\delta}{1+\beta}+\frac{1}{\alpha+1}=1\ \ \mbox{or}\ \ \delta=\frac{1-\alpha\beta}{1+\alpha},$$
then we deduce from the interpolation inequality (see Lemma \ref{dfdasfwqew33}) that
\begin{align}
\mathcal{H}_{12} \leq&C\|\partial_{x_{1}}\theta\|_{L^{2}}^{1-\frac{\delta}{\beta}} \|\Lambda_{x_{2}}^{\beta}\partial_{x_{1}}\theta\|_{L^{2}}^{\frac{\delta}{\beta}}
\|\Lambda_{x_{2}}^{1-\delta}u_{2}\|_{L^{\frac{2(1+\beta)}{1-\delta}}}
\|\partial_{x_{1}}\theta\|_{L^{2(\alpha+1)}}
\nonumber\\
 \leq&C\|\partial_{x_{1}}\theta\|_{L^{2}}^{1-\frac{\delta}{\beta}} \|\Lambda_{x_{2}}^{\beta}\partial_{x_{1}}\theta\|_{L^{2}}^{\frac{\delta}{\beta}}
\|u_{2}\|_{L^{\infty}}^{1-\frac{1-\delta}{1+\beta}}
\|\Lambda_{x_{2}}^{\beta}\partial_{x_{2}}u_{2}\|_{L^{2}}^{\frac{1-\delta}{1+\beta}}
\|\theta\|_{L^{\infty}}^{1-\frac{1}{\alpha+1}}\|\Lambda_{x_{1}}^{\alpha}\partial_{x_{1}}\theta\|_{L^{2}}
^{\frac{1}{\alpha+1}}
\nonumber\\
 \leq&C\|\nabla\theta\|_{L^{2}}^{1-\frac{\delta}{\beta}} \|\Lambda_{x_{2}}^{\beta}\nabla\theta\|_{L^{2}}^{\frac{\delta}{\beta}}
\| u_{2}\|_{L^{\infty}}^{1-\frac{1-\delta}{1+\beta}}
\|\Lambda_{x_{2}}^{\beta}\nabla \theta\|_{L^{2}}^{\frac{1-\delta}{1+\beta}}
\|\theta\|_{L^{\infty}}^{1-\frac{1}{\alpha+1}}\|\Lambda_{x_{1}}^{\alpha}\nabla\theta\|_{L^{2}}
^{\frac{1}{\alpha+1}}
\nonumber\\
 \leq&\epsilon\|\Lambda_{x_{1}}^{\alpha}\nabla\theta\|_{L^{2}}^{2}+\epsilon
\|\Lambda_{x_{2}}^{\beta}\nabla\theta\|_{L^{2}}^{2}+C(\epsilon)
\Big(\|u_{2}\|_{L^{\infty}}^{1-\frac{1-\delta}{1+\beta}}\|\theta\|_{L^{\infty}}
^{1-\frac{1}{\alpha+1}}\Big)^{\frac{2\beta}{\beta-\delta}}
\|\nabla\theta\|_{L^{2}}^{2}
\nonumber\\
 \leq&\epsilon\|\Lambda_{x_{1}}^{\alpha}\nabla\theta\|_{L^{2}}^{2}+\epsilon
\|\Lambda_{x_{2}}^{\beta}\nabla\theta\|_{L^{2}}^{2}+C(\epsilon)
 \|u_{2}\|_{L^{\infty}}^{\frac{2\beta}{(2\alpha+1)\beta-1}}
\|\nabla\theta\|_{L^{2}}^{2}.\nonumber
\end{align}
As a result, the above estimates $\mathcal{H}_{11}$ and $\mathcal{H}_{12}$ would work as long as $\delta$ satisfies
$$1-\beta\leq1-\frac{\alpha(\beta+1)}{\alpha+1}<\beta.$$
The above constraint is in particular satisfied
$$\beta> \frac{1}{2\alpha+1}\ \ \mbox{and}\ \ \beta\geq\alpha.$$
A simple computation shows that
$$\max\Big\{\alpha,\ \frac{1}{2\alpha+1}\Big\}\geq \frac{1}{2}.$$
Substituting the above estimates into (\ref{t3326t003}) yields
\begin{eqnarray}\label{t3326t004}
\mathcal{H}_{1}
 \leq \epsilon\|\Lambda_{x_{1}}^{\alpha}\nabla\theta\|_{L^{2}}^{2}+2\epsilon
\|\Lambda_{x_{2}}^{\beta}\nabla\theta\|_{L^{2}}^{2}+C(\epsilon)
\Big(\|u_{2}\|_{L^{\infty}}^{\frac{2\beta}{2\beta-1}}+\|u_{2} \|_{L^{\infty}}^{\frac{2\beta}{(2\alpha+1)\beta-1}}\Big)
\|\nabla\theta\|_{L^{2}}^{2}.
\end{eqnarray}
Similarly, arguing as the estimates of $\mathcal{H}_{11}$ and $\mathcal{H}_{12}$, we thus have
\begin{align}\label{t3326t005}
\mathcal{H}_{2} =&\int_{\mathbb{R}^{2}}{ \theta\partial_{x_{2}x_{1}}u_{2}\partial_{x_{1}}\theta\,dx}
+\int_{\mathbb{R}^{2}}{ \theta\partial_{x_{1}}u_{2}\partial_{x_{2}x_{1}}\theta\,dx}\nonumber\\
 \leq&C\|\Lambda_{x_{2}}^{\delta}\partial_{x_{1}}u_{2}\|_{L^{2}} \|\Lambda_{x_{2}}^{1-\delta}(\theta \partial_{x_{1}}\theta)\|_{L^{2}}+
C\|\Lambda_{x_{2}}^{\delta}\partial_{x_{1}}\theta\|_{L^{2}} \|\Lambda_{x_{2}}^{1-\delta}(\theta \partial_{x_{1}}u_{2})\|_{L^{2}}
\nonumber\\ \leq&C\|\Lambda_{x_{2}}^{\delta}\partial_{x_{1}}\mathcal {R}_{1}\theta\|_{L^{2}}( \|\theta\|_{L^{\infty}} \|\Lambda_{x_{2}}^{1-\delta}\partial_{x_{1}}\theta\|_{L^{2}}+ \|\Lambda_{x_{2}}^{1-\delta}\theta \partial_{x_{1}}\theta\|_{L^{2}})\nonumber\\&
+C\|\Lambda_{x_{2}}^{\delta}\partial_{x_{1}} \theta\|_{L^{2}}( \|\theta\|_{L^{\infty}} \|\Lambda_{x_{2}}^{1-\delta}\partial_{x_{1}}u_{2}\|_{L^{2}}+ \|\Lambda_{x_{2}}^{1-\delta}\theta \partial_{x_{1}}u_{2}\|_{L^{2}})\nonumber\\
 \leq&\epsilon\|\Lambda_{x_{2}}^{\beta}\nabla\theta\|_{L^{2}}^{2}+C(\epsilon)
\|\theta\|_{L^{\infty}}^{\frac{2\beta}{2\beta-1}} \|\nabla\theta\|_{L^{2}}^{2}+C(\epsilon)
\|\theta\|_{L^{\infty}}^{\frac{2\beta}{\beta-\delta}}
\|\nabla\theta\|_{L^{2}}^{2}\nonumber\\ &+C(\epsilon)
\Big(\|\theta \|_{L^{\infty}}^{1-\frac{1-\delta}{1+\beta}}\|u_{2}\|_{L^{\infty}}
^{1-\frac{1}{\alpha+1}}\Big)^{\frac{2\beta}{\beta-\delta}}
\|\nabla\theta\|_{L^{2}}^{2}
\nonumber\\
 \leq&\epsilon\|\Lambda_{x_{2}}^{\beta}\nabla\theta\|_{L^{2}}^{2}+C(\epsilon)
  \|\nabla\theta\|_{L^{2}}^{2} +C(\epsilon)
 \|u_{2}\|_{L^{\infty}} ^{\frac{2\alpha\beta}{(2\alpha+1)\beta-1}}
\|\nabla\theta\|_{L^{2}}^{2}.
\end{align}
For the term $\mathcal{H}_{3}$, one directly obtains
\begin{align}\label{t3326t006}
\mathcal{H}_{3} =&\int_{\mathbb{R}^{2}}{ u_{1} \partial_{x_{2}x_{1}}\theta\partial_{x_{2}}\theta\,dx}
+\int_{\mathbb{R}^{2}}{ u_{1} \partial_{x_{1}}\theta\partial_{x_{2}x_{2}}\theta\,dx}
\nonumber\\
 \leq&C\|\Lambda_{x_{2}}^{1-\beta}\partial_{x_{1}}\theta\|_{L^{2}} \|\Lambda_{x_{2}}^{\beta}(u_{1} \partial_{x_{2}}\theta)\|_{L^{2}}+
C\|\Lambda_{x_{2}}^{\delta}\partial_{x_{2}}\theta\|_{L^{2}} \|\Lambda_{x_{2}}^{1-\delta}(u_{1} \partial_{x_{1}}\theta)\|_{L^{2}}\nonumber\\
 :=&\mathcal{H}_{31}+\mathcal{H}_{32}.
\end{align}
Applying the same manner dealing with $\mathcal{H}_{11}$ and $\mathcal{H}_{12}$, we immediately get
\begin{align}\label{fg467uilo1}
\mathcal{H}_{32} \leq&C\|\Lambda_{x_{2}}^{\delta}\partial_{x_{2}}\theta\|_{L^{2}} \|\Lambda_{x_{2}}^{1-\delta}(u_{1} \partial_{x_{1}}\theta)\|_{L^{2}}
\nonumber\\
 \leq&C\|\Lambda_{x_{2}}^{\delta}\partial_{x_{2}}\theta\|_{L^{2}}( \|u_{1}\|_{L^{\infty}} \|\Lambda_{x_{2}}^{1-\delta}\partial_{x_{1}}\theta\|_{L^{2}}+ \|\Lambda_{x_{2}}^{1-\delta}u_{1} \partial_{x_{1}}\theta\|_{L^{2}})
\nonumber\\
 \leq&
\epsilon\|\Lambda_{x_{1}}^{\alpha}\nabla\theta\|_{L^{2}}^{2}+2\epsilon
\|\Lambda_{x_{2}}^{\beta}\nabla\theta\|_{L^{2}}^{2}+C(\epsilon)
\Big(\|u_{1}\|_{L^{\infty}}^{\frac{2\beta}{2\beta-1}}+\|u_{2} \|_{L^{\infty}}^{\frac{2\beta}{(2\alpha+1)\beta-1}}\Big)
\|\nabla\theta\|_{L^{2}}^{2}.
\end{align}
For $\beta>\frac{1}{2}$, it follows from  the interpolation inequalities (see Lemma \ref{dfdasfwqew33} and  Lemma \ref{gfhgj}) and the commutator (\ref{t2326003}) that
\begin{align}\label{fg467uilo}
\mathcal{H}_{31} \leq&C\|\Lambda_{x_{2}}^{1-\beta}\partial_{x_{1}}\theta\|_{L^{2}} ( \|u_{1} \|_{L^{\infty}}\|\Lambda_{x_{2}}^{\beta} \partial_{x_{2}}\theta\|_{L^{2}}+\|\Lambda_{x_{2}}^{\beta}u_{1} \partial_{x_{2}}\theta\|_{L^{2}})\nonumber\\
 \leq&C\|\partial_{x_{1}}\theta\|_{L^{2}}^{1-\frac{1-\beta}{\beta}}
 \|\Lambda_{x_{2}}^{\beta}\partial_{x_{1}}\theta\|_{L^{2}}^{\frac{1-\beta}{\beta}}
  \|u_{1} \|_{L^{\infty}}  \|\Lambda_{x_{2}}^{\beta} \partial_{x_{2}}\theta\|_{L^{2}}
\nonumber\\&
 +C\|\partial_{x_{1}}\theta\|_{L^{2}}^{1-\frac{1-\beta}{\beta}}
 \|\Lambda_{x_{2}}^{\beta}\partial_{x_{1}}\theta\|_{L^{2}}^{\frac{1-\beta}{\beta}}
 \|\Lambda_{x_{2}}^{\beta}u_{1}\|_{L^{\frac{2(\beta+1)}{\beta}}}
 \|\partial_{x_{2}}\theta\|_{L^{2(\beta+1)}}
 \nonumber\\
 \leq&C\|\partial_{x_{1}}\theta\|_{L^{2}}^{1-\frac{1-\beta}{\beta}}
 \|\Lambda_{x_{2}}^{\beta}\partial_{x_{1}}\theta\|_{L^{2}}^{\frac{1-\beta}{\beta}}
  \|u_{1} \|_{L^{\infty}}  \|\Lambda_{x_{2}}^{\beta} \partial_{x_{2}}\theta\|_{L^{2}}
 \nonumber\\&
 +C\|\partial_{x_{1}}\theta\|_{L^{2}}^{1-\frac{1-\beta}{\beta}}
 \|\Lambda_{x_{2}}^{\beta}\partial_{x_{1}}\theta\|_{L^{2}}^{\frac{1-\beta}{\beta}}
 \|u_{1}\|_{L^{\infty}}^{1-\frac{\beta}{\beta+1}}
 \|\Lambda_{x_{2}}^{\beta}\partial_{x_{2}}u_{1}
 \|_{L^{2}}^{\frac{\beta}{\beta+1}}
 \|\theta\|_{L^{\infty}}^{1-\frac{1}{\beta+1}}
 \|\Lambda_{x_{2}}^{\beta}\partial_{x_{2}}\theta
 \|_{L^{2}}^{\frac{1}{\beta+1}}
  \nonumber\\
 \leq&C \|u_{1} \|_{L^{\infty}}  \|\nabla\theta\|_{L^{2}}^{2-\frac{1}{\beta}}
 \|\Lambda_{x_{2}}^{\beta}\nabla\theta\|_{L^{2}}^{\frac{1}{\beta}}
 +C\|\nabla\theta\|_{L^{2}}^{2-\frac{1}{\beta}}
 \|\Lambda_{x_{2}}^{\beta}\nabla\theta\|_{L^{2}}^{\frac{1}{\beta}}
 \|u_{1}\|_{L^{\infty}}^{\frac{1}{\beta+1}}
 \|\theta\|_{L^{\infty}}^{\frac{\beta}{\beta+1}}
 \nonumber\\
 \leq&\epsilon\|\Lambda_{x_{2}}^{\beta}\nabla\theta\|_{L^{2}}^{2}+C(\epsilon)
\|u_{1}\|_{L^{\infty}}^{\frac{2\beta}{2\beta-1}} \|\nabla\theta\|_{L^{2}}^{2}+C(\epsilon)\Big( \|u_{1}\|_{L^{\infty}}^{\frac{1}{\beta+1}}
 \|\theta\|_{L^{\infty}}^{\frac{\beta}{\beta+1}}\Big)^{\frac{2\beta}{2\beta-1}}\|\nabla\theta\|_{L^{2}}^{2}
  \nonumber\\
 \leq&\epsilon\|\Lambda_{x_{2}}^{\beta}\nabla\theta\|_{L^{2}}^{2}+C(\epsilon)\Big(
\|u_{1}\|_{L^{\infty}}^{\frac{2\beta}{2\beta-1}} + \|u_{1}\|_{L^{\infty}}^{\frac{2\beta}{(\beta+1)(2\beta-1)}}
 \Big)\|\nabla\theta\|_{L^{2}}^{2}.
\end{align}
Inserting the above two estimates (\ref{fg467uilo1}) and (\ref{fg467uilo}) into (\ref{t3326t006}) yields
\begin{align}\label{t3326t007}
\mathcal{H}_{3} \leq&
\epsilon\|\Lambda_{x_{1}}^{\alpha}\nabla\theta\|_{L^{2}}^{2}+3\epsilon
\|\Lambda_{x_{2}}^{\beta}\nabla\theta\|_{L^{2}}^{2}+C(\epsilon)
\Big(\|u_{1}\|_{L^{\infty}}^{\frac{2\beta}{2\beta-1}}+\|u_{1} \|_{L^{\infty}}^{\frac{2\beta}{(2\alpha+1)\beta-1}}\Big)
\|\nabla\theta\|_{L^{2}}^{2}\nonumber\\&
+C(\epsilon)\Big(
\|u_{1}\|_{L^{\infty}}^{\frac{2\beta}{2\beta-1}} + \|u_{1}\|_{L^{\infty}}^{\frac{2\beta}{(\beta+1)(2\beta-1)}}
 \Big)\|\nabla\theta\|_{L^{2}}^{2}.
\end{align}
Finally, following the estimate of $\mathcal{H}_{31}$, one directly gets for $\beta>\frac{1}{2}$
\begin{align}\label{t3326t008}
\mathcal{H}_{4} =&2\int_{\mathbb{R}^{2}}{ u_{2} \partial_{x_{2}x_{2}}\theta\partial_{x_{2}}\theta\,dx}
\nonumber\\
 \leq&C\|\Lambda_{x_{2}}^{1-\beta}\partial_{x_{2}}\theta\|_{L^{2}} \|\Lambda_{x_{2}}^{\beta}(u_{2} \partial_{x_{2}}\theta)\|_{L^{2}}  \nonumber\\
 \leq&\epsilon\|\Lambda_{x_{2}}^{\beta}\nabla\theta\|_{L^{2}}^{2}+C(\epsilon)\Big(
\|u_{2}\|_{L^{\infty}}^{\frac{2\beta}{2\beta-1}} + \|u_{2}\|_{L^{\infty}}^{\frac{2\beta}{(\beta+1)(2\beta-1)}}
 \Big)\|\nabla\theta\|_{L^{2}}^{2}.
\end{align}
Collecting the estimates (\ref{t3326t002}), (\ref{t3326t004}), (\ref{t3326t005}), (\ref{t3326t007}) and (\ref{t3326t008}), and selecting $\epsilon$ suitable small, it follows that
\begin{eqnarray}\label{t3326t009}
 \frac{d}{dt}\|\nabla \theta(t)\|_{L^{2}}^{2}+
 \|\Lambda_{x_{1}}^{\alpha}\nabla\theta\|_{L^{2}}^{2}+
 \|\Lambda_{x_{2}}^{\beta}\nabla\theta\|_{L^{2}}^{2} \leq H(t)\|\nabla \theta\|_{L^{2}}^{2},
\end{eqnarray}
where
$$H(t)=
C\Big(\|u\|_{L^{\infty}}^{\frac{2\beta}{2\beta-1}}+\|u_{2} \|_{L^{\infty}}^{\frac{2\beta}{(2\alpha+1)\beta-1}}+\|u\|_{L^{\infty}} ^{\frac{2\alpha\beta}{(2\alpha+1)\beta-1}} + \|u\|_{L^{\infty}}^{\frac{2\beta}{(\beta+1)(2\beta-1)}}
 \Big).$$
Obviously, it is easy to show
$$H(t)\leq C\Big(1+\|u\|_{L^{\infty}}^{\varrho}\Big)$$
where
$$\varrho=\max\Big\{\frac{2\beta}{2\beta-1},\, \frac{2\beta}{(2\alpha+1)\beta-1} \Big\}>1.$$
By denoting
$$A(t):=\|\nabla \theta(t)\|_{L^{2}}^{2},\quad B(t):=
 \|\Lambda_{x_{1}}^{\alpha}\nabla\theta(t)\|_{L^{2}}^{2}+
 \|\Lambda_{x_{2}}^{\beta}\nabla\theta(t)\|_{L^{2}}^{2},$$
 we therefore deduce from \eqref{t3326t009} that
\begin{eqnarray}\label{t3326t011}
 \frac{d}{dt}A(t)+
 B(t)\leq C A(t)+C\|u\|_{L^{\infty}}^{\varrho}A(t).
\end{eqnarray}
We deduce by Lemma \ref{gfhgj} that
$$\|\partial_{x_{1}}\theta(t)\|_{L^{2}}\leq C\|\theta(t)\|_{L^{2}}^{1-\frac{1}{\alpha+1}}\|\Lambda_{x_{1}}^{\alpha}
\partial_{x_{1}}\theta(t)\|_{L^{2}}^{\frac{1}{\alpha+1}}\leq C\|\theta_{0}\|_{L^{2}}^{1-\frac{1}{\alpha+1}}\|\Lambda_{x_{1}}^{\alpha}
\nabla\theta(t)\|_{L^{2}}^{\frac{1}{\alpha+1}},$$
$$\|\partial_{x_{2}}\theta(t)\|_{L^{2}}\leq C\|\theta(t)\|_{L^{2}}^{1-\frac{1}{\beta+1}}\|\Lambda_{x_{2}}^{\beta}
\partial_{x_{2}}\theta(t)\|_{L^{2}}^{\frac{1}{\beta+1}}\leq C\|\theta_{0}\|_{L^{2}}^{1-\frac{1}{\beta+1}}\|\Lambda_{x_{2}}^{\beta}
\nabla\theta(t)\|_{L^{2}}^{\frac{1}{\beta+1}}.$$
This further allows us to deduce
\begin{eqnarray}\label{t3326t012}{C}^{-1}A^{\gamma}(t)\leq B (t),\qquad \gamma=\min\{\alpha,\,\beta\}+1>1.\end{eqnarray}
Thanks to Lemma \ref{gfhgj} again, we have
$$\|\partial_{x_{1}}^{\sigma}\theta(t)\|_{L^{2}}\leq C\|\theta(t)\|_{L^{2}}^{1-\frac{\sigma}{\alpha+1}}\|\Lambda_{x_{1}}^{\alpha}
\partial_{x_{1}}\theta(t)\|_{L^{2}}^{\frac{\sigma}{\alpha+1}}\leq C\|\theta_{0}\|_{L^{2}}^{1-\frac{\sigma}{\alpha+1}}\|\Lambda_{x_{1}}^{\alpha}
\nabla\theta(t)\|_{L^{2}}^{\frac{\sigma}{\alpha+1}},$$
$$\|\partial_{x_{2}}^{\sigma}\theta(t)\|_{L^{2}}\leq C\|\theta(t)\|_{L^{2}}^{1-\frac{\sigma}{\beta+1}}\|\Lambda_{x_{2}}^{\beta}
\partial_{x_{2}}\theta(t)\|_{L^{2}}^{\frac{\sigma}{\beta+1}}\leq C\|\theta_{0}\|_{L^{2}}^{1-\frac{\sigma}{\beta+1}}\|\Lambda_{x_{2}}^{\beta}
\nabla\theta(t)\|_{L^{2}}^{\frac{\sigma}{\beta+1}},$$
where $0\leq \sigma\leq \min\{\alpha,\,\beta\}+1$. Now taking some $1<\sigma\leq \min\{\alpha,\,\beta\}+1$, we obtain
\begin{align}\label{t3326t013}
\|\Lambda^{\sigma}\theta(t)\|_{L^{2}} \leq&\|\partial_{x_{2}}^{\sigma}\theta(t)\|_{L^{2}}
+\|\partial_{x_{1}}^{\sigma}\theta(t)\|_{L^{2}}\nonumber\\
 \leq&C(\|\Lambda_{x_{1}}^{\alpha}\nabla\theta(t)\|_{L^{2}} +
 \|\Lambda_{x_{2}}^{\beta}\nabla\theta(t)\|_{L^{2}}),
\end{align}
which leads to
$$\|\Lambda^{\sigma}\theta(t)\|_{L^{2}}\leq e+B(t).$$
Since $\mathcal{R}$ is a bounded operator in homogenous Besov space $\dot{B}_{\infty,\,\infty}^{0}$, this yields
\begin{align}\label{trexzw01}
\|\mathcal {R}f\|_{\dot{B}_{\infty,\,\infty}^{0}}\leq C\|f\|_{\dot{B}_{\infty,\,\infty}^{0}}.
\end{align}
In order to control $\|u\|_{L^{\infty}}$, we need the following logarithmic Sobolev interpolation inequality (see for example \cite{KKozonota})
\begin{align}\label{trexzw02}
\|f\|_{L^{\infty}}\leq C\left(1+\|f\|_{L^{2}}+\|f\|_{\dot{B}_{\infty,\,\infty}^{0}}\ln \big(e+\|\Lambda^{\sigma}f\|_{L^{2}} \big)\right),\quad \forall \sigma>1.
\end{align}
Hence it follows from (\ref{t3326t011}), \eqref{t3326t012}, (\ref{t3326t013}), \eqref{trexzw01} and \eqref{trexzw02} that
\begin{align}
 \frac{d}{dt}A(t)+
 B(t) \leq& C A(t)+C\|u(t)\|_{\dot{B}_{\infty,\,\infty}^{0}}^{\varrho}\Big(\ln\big(e
 +\|\Lambda^{\sigma}u(t)\|_{L^{2}} \big)\Big)^{\varrho}A(t)\nonumber\\
 \leq& C A(t)+C\|\mathcal {R}^{\perp}\theta(t)\|_{\dot{B}_{\infty,\,\infty}^{0}}^{\varrho}\Big(\ln\big(e+\|\Lambda^{\sigma}\mathcal {R}^{\perp}\theta(t)\|_{L^{2}} \big)\Big)^{\varrho}A(t)\nonumber\\
 \leq& C A(t)+C\|\theta(t)\|_{\dot{B}_{\infty,\,\infty}^{0}}^{\varrho}\Big(\ln\big(e+\|\Lambda^{\sigma} \theta(t)\|_{L^{2}} \big)\Big)^{\varrho}A(t)
 \nonumber\\
 \leq& C A(t)+C\|\theta(t)\|_{L^{\infty}}^{\varrho}\Big(\ln\big(e+\|\Lambda^{\sigma} \theta(t)\|_{L^{2}} \big)\Big)^{\varrho}A(t)
 \nonumber\\
 \leq& C A(t)+C\|\theta_{0}\|_{L^{\infty}}^{\varrho}\Big(\ln\big(e+B(t) \big)\Big)^{\varrho}A(t),\nonumber
\end{align}
where we have used the embedding $L^{\infty}\hookrightarrow \dot{B}_{\infty,\,\infty}^{0}$. In fact, the embedding $L^{\infty}\hookrightarrow \dot{B}_{\infty,\,\infty}^{0}$ can be deduced by
$$\|\theta\|_{\dot{B}_{\infty,\,\infty}^{0}}=\sup_{j\in\mathbb{Z}}
\|\dot{\Delta}_{j}\theta\|_{L^{\infty}}\leq C\sup_{j\in\mathbb{Z}}
\|\theta\|_{L^{\infty}}=C\|\theta\|_{L^{\infty}}.$$
We finally get
\begin{eqnarray}\label{t3326t014}
 \frac{d}{dt}A(t)+
 B(t)
 \leq C \big(A(t)+e\big)+C\|\theta_{0}\|_{L^{\infty}}^{\varrho}\Big(\ln\big(A(t)+B(t)+e \big)\Big)^{\varrho}\big(A(t)+e\big).
\end{eqnarray}
Applying the logarithmic type Gronwall
inequality (see Lemma \ref{addLdem01}) to (\ref{t3326t014}), we therefore obtain
$$A(t)+\int_{0}^{t}{B(s)\,ds}\leq C,$$
which is nothing but the desired estimate (\ref{asddfgdfht302}). Consequently, we complete the proof of Proposition \ref{Lpp302}.
\end{proof}
\vskip .1in
Next we will prove the global
$H^1$-bound for $\beta>\frac{1-\alpha}{2\alpha}$ and $\alpha>\frac{1}{2}$.
\begin{Pros}\label{Lpp303}
Assume $\theta_{0}$ satisfies the assumptions stated in Theorem \ref{ThSQG} and let $(u, \theta)$ be the corresponding solution. If $\alpha$ and $\beta$ satisfy
\begin{eqnarray}
\beta>\frac{1-\alpha}{2\alpha} \ \ \mbox{and}\ \  \alpha>\frac{1}{2},\nonumber
\end{eqnarray}
then, for any $t>0$,
\begin{eqnarray}\label{asd556fgdfsd}
\|\nabla\theta(t)\|_{L^{2}}^{2}+ \int_{0}^{t}{(\|\Lambda_{x_{1}}^{\alpha}\nabla\theta(
\tau)\|_{L^{2}}^{2}+\|\Lambda_{x_{2}}^{\beta}\nabla\theta(
\tau)\|_{L^{2}}^{2})\,d\tau}\leq
C(t,\,\theta_{0}),
\end{eqnarray}
where $C(t, \,\theta_{0})$ is a constant depending on $t$ and the initial data $\theta_0$.
\end{Pros}

\begin{rem}\rm
In this case, our main target is focused on the $\alpha$ and $\beta$ satisfying
\begin{eqnarray}\label{cvbnyutyu23}
\alpha>\beta>\frac{1-\alpha}{2\alpha}\ \ \mbox{and}\ \  \alpha>\frac{1}{2}.
\end{eqnarray}
The upper bound restriction on $\beta$, namely $\beta<\alpha$ is actually a technical
assumption. In common sense, it is commonly believed that the diffusion term is always good term and the
larger the power $\beta$ is, the better effects it produces. As a matter of fact, if $\beta\geq\alpha$ and $\alpha>\frac{1}{2}$, then Proposition \ref{Lpp303} is a direct consequence of Proposition \ref{Lpp302}.
\end{rem}

\begin{proof}[{Proof of Proposition \ref{Lpp303}}]
It follows from (\ref{t3326t002}) that
\begin{eqnarray}\label{t3326t016}
\frac{1}{2}\frac{d}{dt}\|\nabla \theta(t)\|_{L^{2}}^{2}+
 \|\Lambda_{x_{1}}^{\alpha}\nabla\theta\|_{L^{2}}^{2}+
 \|\Lambda_{x_{2}}^{\beta}\nabla\theta\|_{L^{2}}^{2}= \mathcal{H}_{1}+\mathcal{H}_{2}+\mathcal{H}_{3}+\mathcal{H}_{4}.
\end{eqnarray}
By means of the commutator (\ref{t2326003}), it ensures for $\alpha>\frac{1}{2}$
\begin{align} \label{tsdgdsgu8}
\mathcal{H}_{1} =&-\int_{\mathbb{R}^{2}}{\partial_{x_{1}}u_{1} \partial_{x_{1}}\theta\partial_{x_{1}}\theta\,dx}\nonumber\\
 =&2\int_{\mathbb{R}^{2}}{u_{1} \partial_{x_{1}}\theta \partial_{x_{1}x_{1}}\theta\,dx}
\nonumber\\
 \leq&C\|\Lambda_{x_{1}}^{1-\alpha}\partial_{x_{1}}\theta\|_{L^{2}} \|\Lambda_{x_{1}}^{\alpha}(u_{1} \partial_{x_{1}}\theta)\|_{L^{2}}
\nonumber\\
 \leq&C\|\Lambda_{x_{1}}^{1-\alpha}\partial_{x_{1}}\theta\|_{L^{2}}( \|u_{1}\|_{L^{\infty}} \|\Lambda_{x_{1}}^{\alpha}\partial_{x_{1}}\theta\|_{L^{2}}+ \|\Lambda_{x_{1}}^{\alpha}u_{1} \partial_{x_{1}}\theta\|_{L^{2}})
\nonumber\\
 \leq&C\|\partial_{x_{1}}\theta\|_{L^{2}}^{1-\frac{1-\alpha}{\alpha}}\|\Lambda_{x_{1}}^{\alpha}\partial_{x_{1}}
\theta\|_{L^{2}}^{\frac{1-\alpha}{\alpha}}( \|u_{1}\|_{L^{\infty}} \|\Lambda_{x_{1}}^{\alpha}\partial_{x_{1}}\theta\|_{L^{2}}+ \|\Lambda_{x_{1}}^{\alpha}u_{1}\|_{L^{\frac{2(\alpha+1)}{\alpha}}} \| \partial_{x_{1}}\theta\|_{L^{2(\alpha+1)}})
\nonumber\\
 \leq&C\|\partial_{x_{1}}\theta\|_{L^{2}}^{1-\frac{1-\alpha}{\alpha}}
 \|\Lambda_{x_{1}}^{\alpha}\partial_{x_{1}}
\theta\|_{L^{2}}^{\frac{1-\alpha}{\alpha}}  \|u_{1}\|_{L^{\infty}} \|\Lambda_{x_{1}}^{\alpha}\partial_{x_{1}}\theta\|_{L^{2}} \nonumber\\& +
C\|\partial_{x_{1}}\theta\|_{L^{2}}^{1-\frac{1-\alpha}{\alpha}}
\|\Lambda_{x_{1}}^{\alpha}\partial_{x_{1}}
\theta\|_{L^{2}}^{\frac{1-\alpha}{\alpha}} \|u_{1}\|_{L^{\infty}}^{1-\frac{\alpha}{\alpha+1}}
\|\Lambda_{x_{1}}^{\alpha}\partial_{x_{1}}u_{1}\|_{L^{2}}^{\frac{\alpha}{\alpha+1}} \|\theta\|_{L^{\infty}}^{1-\frac{1}{\alpha+1}}
\|\Lambda_{x_{1}}^{\alpha}\partial_{x_{1}}\theta\|_{L^{2}}^{\frac{1}{\alpha+1}}
\nonumber\\
 \leq&\epsilon\|\Lambda_{x_{1}}^{\alpha}\nabla\theta\|_{L^{2}}^{2}+C(\epsilon)
\|u \|_{L^{\infty}}^{\frac{2\alpha}{2\alpha-1}} \|\nabla\theta\|_{L^{2}}^{2}+C(\epsilon)
\|u \|_{L^{\infty}}^{\frac{2\alpha}{(\alpha+1)(2\alpha-1)}}\|\nabla\theta\|_{L^{2}}^{2}.
\end{align}
We rewrite $\mathcal{H}_{2}$ as
\begin{eqnarray}
\mathcal{H}_{2}=\int_{\mathbb{R}^{2}}{ u_{2} \partial_{x_{1}x_{2}}\theta\partial_{x_{1}}\theta\,dx}
+\int_{\mathbb{R}^{2}}{ u_{2} \partial_{x_{2}}\theta\partial_{x_{1}x_{1}}\theta\,dx}:=\mathcal{H}_{21}+\mathcal{H}_{22}.\nonumber
\end{eqnarray}
According to the estimate of (\ref{tsdgdsgu8}), we infer that
\begin{align}
\mathcal{H}_{21} \leq&C\|\Lambda_{x_{1}}^{1-\alpha}\partial_{x_{2}}\theta\|_{L^{2}} \|\Lambda_{x_{1}}^{\alpha}(u_{2} \partial_{x_{1}}\theta)\|_{L^{2}}\nonumber\\
 \leq&\epsilon\|\Lambda_{x_{1}}^{\alpha}\nabla\theta\|_{L^{2}}^{2}+C(\epsilon)
\|u \|_{L^{\infty}}^{\frac{2\alpha}{2\alpha-1}} \|\nabla\theta\|_{L^{2}}^{2}+C(\epsilon)
\|u\|_{L^{\infty}}^{\frac{2\alpha}{(\alpha+1)(2\alpha-1)}}\|\nabla\theta\|_{L^{2}}^{2}.\nonumber
\end{align}
As $\alpha$ and $\beta$ satisfy the condition (\ref{cvbnyutyu23}), we may choose $\widetilde{\delta} \in (1-\alpha,\,\alpha)$ as
$$\frac{\widetilde{\delta}}{\alpha+1}+\frac{1}{\beta+1}=1\quad \mbox{or}\quad\widetilde{\delta}=\frac{\alpha+1}{\beta+1}\beta.$$
Now the term $\mathcal{H}_{22}$ can be estimated as follows
\begin{align}\label{tsdvvdfwert}
\mathcal{H}_{22} \leq&C\|\Lambda_{x_{1}}^{1-\widetilde{\delta}}\partial_{x_{1}}\theta\|_{L^{2}} \|\Lambda_{x_{1}}^{\widetilde{\delta}}(u_{2} \partial_{x_{2}}\theta)\|_{L^{2}}\nonumber\\
 \leq& C\|\Lambda_{x_{1}}^{1-\widetilde{\delta}}\partial_{x_{1}}\theta\|_{L^{2}} (\|u_{2}\|_{L^{\infty}}\|\Lambda_{x_{1}}^{\widetilde{\delta}}\partial_{x_{2}}\theta\|_{L^{2}}
+\|\Lambda_{x_{1}}^{\widetilde{\delta}}u_{2} \partial_{x_{2}}\theta\|_{L^{2}})\nonumber\\
 \leq& C\|\Lambda_{x_{1}}^{1-\widetilde{\delta}}\partial_{x_{1}}\theta\|_{L^{2}} (\|u_{2}\|_{L^{\infty}}\|\Lambda_{x_{1}}^{\widetilde{\delta}}
 \partial_{x_{2}}\theta\|_{L^{2}}
+\|\Lambda_{x_{1}}^{\widetilde{\delta}}u_{2}
\|_{L^{\frac{2(\alpha+1)}{\widetilde{\delta}}}}
\|\partial_{x_{2}}\theta\|_{L^{2(\beta+1)}})
\nonumber\\
 \leq& C\|\partial_{x_{1}}\theta\|_{L^{2}}^{1-\frac{1-\widetilde{\delta}}{\alpha}}
\|\Lambda_{x_{1}}^{\alpha}\partial_{x_{1}}\theta\|_{L^{2}}^{\frac{1-\widetilde{\delta}}{\alpha}}  \|u_{2}\|_{L^{\infty}}\|\partial_{x_{2}}\theta\|_{L^{2}}
^{1-\frac{\widetilde{\delta}}{\alpha}}
\|\Lambda_{x_{1}}^{\alpha}\partial_{x_{2}}\theta\|_{L^{2}}
^{\frac{\widetilde{\delta}}{\alpha}}\nonumber\\&
+C\| \partial_{x_{1}}\theta\|_{L^{2}}^{1-\frac{1-\widetilde{\delta}}{\alpha}}
\|\Lambda_{x_{1}}^{\alpha}\partial_{x_{1}}\theta\|_{L^{2}}
^{\frac{1-\widetilde{\delta}}{\alpha}}
\|u_{2}\|_{L^{\infty}}^{1-\frac{\widetilde{\delta}}{\alpha+1}}
\|\Lambda_{x_{1}}^{\alpha}\partial_{x_{1}}u_{2}\|_{L^{2}}
^{\frac{\widetilde{\delta}}{\alpha+1}}
\|\theta\|_{L^{\infty}}^{1-\frac{1}{\beta+1}}
\|\Lambda_{x_{2}}^{\beta}\partial_{x_{2}}\theta\|_{L^{2}}^{\frac{1}{\beta+1}}
\nonumber\\
 \leq&\epsilon\|\Lambda_{x_{1}}^{\alpha}\nabla\theta\|_{L^{2}}^{2}+
 \epsilon\|\Lambda_{x_{2}}^{\beta}
\nabla\theta\|_{L^{2}}^{2}+C(\epsilon)\Big(
\|u \|_{L^{\infty}}^{\frac{2\alpha}{2\alpha-1}}+\|u \|_{L^{\infty}}^{\frac{2\alpha}{2\alpha\beta+\alpha-1}}\Big) \|\nabla\theta\|_{L^{2}}^{2}.
\end{align}
Similar to (\ref{tsdvvdfwert}), it is also clear that
\begin{align}
\mathcal{H}_{3} =&\int_{\mathbb{R}^{2}}{\theta\partial_{x_{1}x_{2}}u_{1}\partial_{x_{2}}\theta\,dx}
+\int_{\mathbb{R}^{2}}{\theta\partial_{x_{2}}u_{1}\partial_{x_{1}x_{2}}\theta\,dx}
\nonumber\\
 \leq&C\|\Lambda_{x_{1}}^{1-\widetilde{\delta}}\partial_{x_{2}}u_{1}\|_{L^{2}} \|\Lambda_{x_{1}}^{\widetilde{\delta}}(\theta\partial_{x_{2}}\theta)\|_{L^{2}}
+C\|\Lambda_{x_{1}}^{1-\widetilde{\delta}}\partial_{x_{2}}\theta\|_{L^{2}} \|\Lambda_{x_{1}}^{\widetilde{\delta}}(\theta\partial_{x_{2}}u_{1})\|_{L^{2}}
\nonumber\\
 \leq&\epsilon\|\Lambda_{x_{1}}^{\alpha}\nabla\theta\|_{L^{2}}^{2}
 +\epsilon\|\Lambda_{x_{2}}^{\beta}
\nabla\theta\|_{L^{2}}^{2}+C(\epsilon)\|\nabla\theta\|_{L^{2}}^{2}
\nonumber\\& +C(\epsilon) \|u\|_{L^{\infty}}^{\frac{2\alpha\beta}{2\alpha\beta+\alpha-1}} \|\theta\|_{L^{\infty}}^{\frac{2\alpha}{2\alpha\beta+\alpha-1}} \|\nabla\theta\|_{L^{2}}^{2}
\nonumber\\
 \leq&\epsilon\|\Lambda_{x_{1}}^{\alpha}\nabla\theta\|_{L^{2}}^{2}
 +\epsilon\|\Lambda_{x_{2}}^{\beta}
\nabla\theta\|_{L^{2}}^{2}+C(\epsilon)\|\nabla\theta\|_{L^{2}}^{2} +C(\epsilon)
\|u\|_{L^{\infty}}^{\frac{2\alpha\beta}{2\alpha\beta+\alpha-1}}  \|\nabla\theta\|_{L^{2}}^{2}.\nonumber
\end{align}
Repeating the argument used in proving (\ref{tsdvvdfwert}), one obtains
\begin{align}
\mathcal{H}_{4} =&\int_{\mathbb{R}^{2}}{\partial_{x_{1}}u_{1} \partial_{x_{2}}\theta\partial_{x_{2}}\theta\,dx}
\nonumber\\
 =&-2
\int_{\mathbb{R}^{2}}{ u_{1} \partial_{x_{2}}\theta\partial_{x_{1}x_{2}}\theta\,dx}
\nonumber\\
 \leq&
C\|\Lambda_{x_{1}}^{1-\widetilde{\delta}}\partial_{x_{2}}\theta\|_{L^{2}} \|\Lambda_{x_{1}}^{\widetilde{\delta}}(u_{1}\partial_{x_{2}}\theta)\|_{L^{2}}
\nonumber\\
 \leq&\epsilon\|\Lambda_{x_{1}}^{\alpha}\nabla\theta\|_{L^{2}}^{2}
 +\epsilon\|\Lambda_{x_{2}}^{\beta}
\nabla\theta\|_{L^{2}}^{2}+C(\epsilon)\Big(
\|u \|_{L^{\infty}}^{\frac{2\alpha}{2\alpha-1}}+ \|u \|_{L^{\infty}}^{\frac{2\alpha}{2\alpha\beta+\alpha-1}} \Big) \|\nabla\theta\|_{L^{2}}^{2}.\nonumber
\end{align}
Putting the above estimates into \eqref{t3326t016} and taking $\epsilon$ sufficiently small, it turns out that
\begin{eqnarray}
 \frac{d}{dt}\|\nabla \theta(t)\|_{L^{2}}^{2}+
 \|\Lambda_{x_{1}}^{\alpha}\nabla\theta\|_{L^{2}}^{2}+
 \|\Lambda_{x_{2}}^{\beta}\nabla\theta\|_{L^{2}}^{2} \leq \widetilde{H}(t)\|\nabla \theta\|_{L^{2}}^{2},\nonumber
\end{eqnarray}
where
$$\widetilde{H}(t)=
C\Big(\|u\|_{L^{\infty}}^{\frac{2\alpha}{2\alpha-1}}+\|u \|_{L^{\infty}}^{\frac{2\alpha}{(\alpha+1)(2\alpha-1)}} + \|u \|_{L^{\infty}}^{\frac{2\alpha}{2\alpha\beta+\alpha-1}}
+ \|u\|_{L^{\infty}}^{\frac{2\alpha\beta}{2\alpha\beta+\alpha-1}}\Big).$$
Obviously, we have
$$\widetilde{H}(t)\leq C\Big(1+\|u\|_{L^{\infty}}^{\widetilde{\varrho}}\Big)$$
where
$$\widetilde{\varrho}=\max\Big\{\frac{2\alpha}{2\alpha-1},\,\frac{2\alpha}{2\alpha\beta+\alpha-1}\Big\}>1.$$
Finally, the left part of the proof of Proposition \ref{Lpp303} proceeds
by the same manner as that of Proposition \ref{Lpp302}.
 In order to avoid redundancy, the details are omitted here. This completes the proof of Proposition \ref{Lpp303}.
\end{proof}

\vskip .1in
With the global
$H^1$-bound of $\theta$ at our disposal, we will establish the global
$H^2$-bound.
\begin{Pros}\label{Lp5t12}
Assume $\theta_{0}$ satisfies the assumptions stated in Theorem \ref{ThSQG} and let $(u, \theta)$ be the corresponding solution. If $\alpha$ and $\beta$ satisfy (\ref{sdf2334}),
then, for any $t>0$,
\begin{eqnarray}\label{t3326t017}
\|\Delta\theta(t)\|_{L^{2}}^{2}+ \int_{0}^{t}{(\|\Lambda_{x_{1}}^{\alpha}\Delta\theta(
\tau)\|_{L^{2}}^{2}+\|\Lambda_{x_{2}}^{\beta}\Delta\theta(
\tau)\|_{L^{2}}^{2})\,d\tau}\leq
C(t,\,\theta_{0}),
\end{eqnarray}
where $C(t, \,\theta_{0})$ is a constant depending on $t$ and the initial data $\theta_0$.
\end{Pros}

\begin{proof}[{Proof of Proposition \ref{Lp5t12}}]
Applying $\Delta$
to the first equation of $(\ref{SQG})$, multiplying the resulting identity by $\Delta \theta$ and integrating over $\mathbb{R}^{2}$ by parts, we immediately deduce that
\begin{eqnarray}\label{t3326t018}
\frac{1}{2}\frac{d}{dt}\|\Delta\theta(t)\|_{L^{2}}^{2}+\|\Lambda_{x_{1}}^{\alpha}\Delta\theta\|_{L^{2}}^{2}+
 \|\Lambda_{x_{2}}^{\beta}\Delta\theta\|_{L^{2}}^{2}
=-\int_{\mathbb{R}^{2}}{\Delta\{(u\cdot\nabla)\theta\} \Delta \theta\,dx}.
\end{eqnarray}
Using the divergence free condition, the term at the right hand side of (\ref{t3326t018}) can be rewritten as
\begin{align}\label{t3326t019}
& -\int_{\mathbb{R}^{2}}{\Delta\{(u\cdot\nabla)\theta\} \Delta \theta\,dx}
\nonumber\\
&=  \int_{\mathbb{R}^{2}}{\Delta u_{1} \partial_{x_{1}}\theta \Delta \theta\,dx}+\int_{\mathbb{R}^{2}}{ \Delta u_{2} \partial_{x_{2}}\theta  \Delta \theta\,dx}
+2\int_{\mathbb{R}^{2}}{\partial_{x_{1}}u_{1}\partial_{x_{1}x_{1}}\theta \Delta \theta\,dx}\nonumber\\& \quad +2\int_{\mathbb{R}^{2}}{\partial_{x_{2}}u_{1}\partial_{x_{1}x_{2}}
\theta \Delta \theta\,dx}+2\int_{\mathbb{R}^{2}}{\partial_{x_{1}}u_{2}\partial_{x_{1}x_{2}}\theta \Delta \theta\,dx}
+2
\int_{\mathbb{R}^{2}}{\partial_{x_{2}}u_{2}\partial_{x_{2}x_{2}}\theta\Delta \theta\,dx}\nonumber\\
&:=  \mathcal{T}_{1}+\mathcal{T}_{2}+\cdot\cdot\cdot+\mathcal{T}_{6}.
\end{align}
Our next goal is to handle the six terms at the right hand side of (\ref{t3326t019}). Let us first notice some basic estimates.
Due to Plancherel's Theorem and the following simple inequality
$$
|\xi_2| ^{2\alpha}|\xi_1|^{2}  \leq  |\xi_1|^{2\alpha}|\xi|^{2},
$$
we arrive at
\begin{eqnarray}\label{tuset001}\|\Lambda_{x_{2}}^{\alpha}\partial_{x_{1}}\theta\|_{L^2}\leq \|\Lambda_{x_{1}}^{\alpha}\nabla\theta\|_{L^2}.\end{eqnarray}
Keeping in mind the fact $u=(-\mathcal
{R}_{2}\theta,\,\,\mathcal {R}_{1}\theta)$ and using the same argument adopted in proving (\ref{tuset001}), one may conclude the following estimates which will be needed to estimate the terms $\mathcal{T}_{1}-\mathcal{T}_{6}$
\begin{eqnarray}\label{tuset002}\|\Lambda_{x_{1}}^{\beta}\Delta u_{1}\|_{L^2}=\|\Lambda_{x_{1}}^{\beta}\Delta \mathcal{R}_{2}\theta\|_{L^2}
\leq\|\Lambda_{x_{2}}^{\beta}\Delta  \theta\|_{L^2},
\end{eqnarray}
\begin{eqnarray}\label{tuset003}\|\Lambda_{x_{2}}^{\alpha}\Delta u_{2}\|_{L^2}^{ \frac{1}{2\alpha}}\leq \|\Lambda_{x_{2}}^{\alpha}\Delta \mathcal{R}_{1}\theta\|_{L^2}^{ \frac{1}{2\alpha}}\leq\|\Lambda_{x_{1}}^{\alpha}\Delta  \theta\|_{L^2},
\end{eqnarray}
\begin{eqnarray}\label{tuset004}
\|\Lambda_{x_{2}}^{\alpha}\partial_{x_{1}x_{1}}\theta\|_{L^2}\leq \|\Lambda_{x_{1}}^{\alpha}\Delta  \theta\|_{L^2},
\end{eqnarray}
\begin{eqnarray}\label{tuset005}
\|\Lambda_{x_{1}}^{\beta}\partial_{x_{1}}u_{1}\|_{L^2}\leq \|\Lambda_{x_{1}}^{\beta}\partial_{x_{1}}\mathcal{R}_{2}\theta\|_{L^2}\leq \|\Lambda_{x_{2}}^{\beta}\nabla\theta\|_{L^2},
\end{eqnarray}
\begin{eqnarray}\label{tuset006}
\|\Lambda_{x_{2}}^{\alpha}\partial_{x_{1}x_{2}}
\theta\|_{L^2}\leq \|\Lambda_{x_{1}}^{\alpha}\Delta  \theta\|_{L^2},
\end{eqnarray}
\begin{eqnarray}\label{tuset007}
\|\Lambda_{x_{1}}^{\beta}\partial_{x_{2}}u_{1}\|_{L^2}\leq \|\Lambda_{x_{1}}^{\beta}\partial_{x_{2}}\theta\|_{L^2}\leq \|\Lambda_{x_{2}}^{\beta}\nabla\theta\|_{L^2},
\end{eqnarray}
\begin{eqnarray}\label{tuset008}
\|\Lambda_{x_{2}}^{\alpha}\partial_{x_{1}x_{2}}\theta\|_{L^2}\leq \|\Lambda_{x_{1}}^{\alpha}\Delta  \theta\|_{L^2},
\end{eqnarray}
\begin{eqnarray}\label{tuset009}
\|\Lambda_{x_{1}}^{\beta}\partial_{x_{1}x_{2}}\theta\|_{L^2}\leq\|\Lambda_{x_{2}}^{\beta}\Delta \theta\|_{L^2},
\end{eqnarray}
\begin{eqnarray}\label{tuset010}
\|\Lambda_{x_{2}}^{\alpha}\partial_{x_{2}}u_{2}\|_{L^2}\leq \|\Lambda_{x_{2}}^{\alpha}\partial_{x_{2}}\mathcal{R}_{1}\theta\|_{L^2}\leq\|\Lambda_{x_{1}}^{\alpha}\nabla \theta\|_{L^2},
\end{eqnarray}
\begin{eqnarray}\label{tuset011}
\|\Lambda_{x_{1}}^{\beta}\partial_{x_{2}x_{2}}\theta\|_{L^2}\leq\|\Lambda_{x_{2}}^{\beta}\Delta \theta\|_{L^2}.
\end{eqnarray}
It should be mentioned that if $\alpha$ and $\beta$ satisfy (\ref{sdf2334}), then $\alpha>\frac{1}{2}$ or $\beta>\frac{1}{2}$ holds true. Therefore, we split the proof into two cases, namely,
$$\mbox{\textbf{Case 1}}:\  \alpha>\frac{1}{2};\qquad \mbox{\textbf{Case 2}}:\  \beta>\frac{1}{2}.$$
For the \textbf{Case 1}, the inequality (\ref{qtri}) implies the following bounds
\begin{align}
\mathcal{T}_{1}
 \leq& C\|\Delta \theta\|_{L^2} \, \|\partial_{x_{1}}\theta\|_{L^2}^{1-\frac{1}{2\alpha}}\|\Lambda_{x_{2}}^{\alpha}\partial_{x_{1}}\theta\|_{L^2}^{ \frac{1}{2\alpha}} \,\|\Delta u_{1}\|_{L^2}^{1-\frac{1}{2\alpha}}\|\Lambda_{x_{1}}^{\alpha}\Delta u_{1}\|_{L^2}^{ \frac{1}{2\alpha}}\nonumber\\
 \leq& C\|\Delta \theta\|_{L^2} \, \|\nabla\theta\|_{L^2}^{1-\frac{1}{2\alpha}}\|\Lambda_{x_{1}}^{\alpha}\nabla\theta\|_{L^2}^{ \frac{1}{2\alpha}} \,\|\Delta \theta\|_{L^2}^{1-\frac{1}{2\alpha}}\|\Lambda_{x_{1}}^{\alpha}\Delta \theta\|_{L^2}^{ \frac{1}{2\alpha}}\quad \Big(\mbox{using} \ (\ref{tuset001})\Big)
\nonumber\\
 \leq&
\epsilon\|\Lambda_{x_{1}}^{\alpha}\Delta\theta\|_{L^{2}}^{2}+C(\epsilon)
\|\nabla\theta\|_{L^2}^{\frac{2(2\alpha-1)}{4\alpha-1}}\|\Lambda_{x_{1}}^{\alpha}\nabla\theta\|_{L^2}^{ \frac{2}{4\alpha-1}} \,\|\Delta \theta\|_{L^2}^{2}\nonumber\\
 \leq&
\epsilon\|\Lambda_{x_{1}}^{\alpha}\Delta\theta\|_{L^{2}}^{2}+C(\epsilon)
\|\nabla\theta\|_{L^2}^{\frac{2(2\alpha-1)}{4\alpha-1}}
(1+\|\Lambda_{x_{1}}^{\alpha}\nabla\theta\|_{L^2}^{ 2} ) \,\|\Delta \theta\|_{L^2}^{2},\nonumber
\end{align}
\begin{align}
\mathcal{T}_{2} 
 \leq& C\|\Delta \theta\|_{L^2} \, \|\partial_{x_{2}}\theta\|_{L^2}^{1-\frac{1}{2\alpha}}
 \|\Lambda_{x_{1}}^{\alpha}\partial_{x_{2}}\theta\|_{L^2}^{ \frac{1}{2\alpha}} \,\|\Delta u_{2}\|_{L^2}^{1-\frac{1}{2\alpha}}\|\Lambda_{x_{2}}^{\alpha}\Delta u_{2}\|_{L^2}^{ \frac{1}{2\alpha}}\nonumber\\
 \leq& C\|\Delta \theta\|_{L^2} \, \|\nabla\theta\|_{L^2}^{1-\frac{1}{2\alpha}}\|\Lambda_{x_{1}}^{\alpha}\nabla\theta\|_{L^2}^{ \frac{1}{2\alpha}} \,\|\Delta \theta\|_{L^2}^{1-\frac{1}{2\alpha}}\|\Lambda_{x_{1}}^{\alpha}\Delta \theta\|_{L^2}^{ \frac{1}{2\alpha}}\quad \Big(\mbox{using} \ (\ref{tuset003})\Big)
\nonumber\\
 \leq&
\epsilon\|\Lambda_{x_{1}}^{\alpha}\Delta\theta\|_{L^{2}}^{2}+C(\epsilon)
\|\nabla\theta\|_{L^2}^{\frac{2(2\alpha-1)}{4\alpha-1}}
(1+\|\Lambda_{x_{1}}^{\alpha}\nabla\theta\|_{L^2}^{ 2} ) \,\|\Delta \theta\|_{L^2}^{2},\nonumber
\end{align}
\begin{align}
\mathcal{T}_{3} 
 \leq& C\|\Delta \theta\|_{L^2} \, \|\partial_{x_{1}}u_{1}\|_{L^2}^{1-\frac{1}{2\alpha}}
 \|\Lambda_{x_{1}}^{\alpha}\partial_{x_{1}}u_{1}\|_{L^2}^{ \frac{1}{2\alpha}} \,\|\partial_{x_{1}x_{1}}\theta\|_{L^2}^{1-\frac{1}{2\alpha}}
\|\Lambda_{x_{2}}^{\alpha}\partial_{x_{1}x_{1}}\theta\|_{L^2}^{ \frac{1}{2\alpha}}\nonumber\\
 \leq& C\|\Delta \theta\|_{L^2} \, \|\nabla\theta\|_{L^2}^{1-\frac{1}{2\alpha}}
 \|\Lambda_{x_{1}}^{\alpha}\nabla\theta\|_{L^2}^{ \frac{1}{2\alpha}} \,\|\Delta \theta\|_{L^2}^{1-\frac{1}{2\alpha}}\|\Lambda_{x_{1}}^{\alpha}\Delta \theta\|_{L^2}^{ \frac{1}{2\alpha}}\quad \Big(\mbox{using} \ (\ref{tuset004})\Big)
\nonumber\\
 \leq&
\epsilon\|\Lambda_{x_{1}}^{\alpha}\Delta\theta\|_{L^{2}}^{2}+C(\epsilon)
\|\nabla\theta\|_{L^2}^{\frac{2(2\alpha-1)}{4\alpha-1}}(1+\|\Lambda_{x_{1}}^{\alpha}\nabla\theta\|_{L^2}^{ 2} ) \,\|\Delta \theta\|_{L^2}^{2},\nonumber
\end{align}
\begin{align}
\mathcal{T}_{4} 
 \leq& C\|\Delta \theta\|_{L^2} \, \|\partial_{x_{2}}u_{1}\|_{L^2}^{1-\frac{1}{2\alpha}}\|\Lambda_{x_{1}}^{\alpha}\partial_{x_{2}}u_{1}\|_{L^2}^{ \frac{1}{2\alpha}} \,\|\partial_{x_{1}x_{2}}
\theta\|_{L^2}^{1-\frac{1}{2\alpha}}
\|\Lambda_{x_{2}}^{\alpha}\partial_{x_{1}x_{2}}
\theta\|_{L^2}^{ \frac{1}{2\alpha}}\nonumber\\
 \leq& C\|\Delta \theta\|_{L^2} \, \|\nabla\theta\|_{L^2}^{1-\frac{1}{2\alpha}}\|\Lambda_{x_{1}}^{\alpha}\nabla\theta\|_{L^2}^{ \frac{1}{2\alpha}} \,\|\Delta \theta\|_{L^2}^{1-\frac{1}{2\alpha}}\|\Lambda_{x_{1}}^{\alpha}\Delta \theta\|_{L^2}^{ \frac{1}{2\alpha}}\quad \Big(\mbox{using} \ (\ref{tuset006})\Big)
\nonumber\\
 \leq&
\epsilon\|\Lambda_{x_{1}}^{\alpha}\Delta\theta\|_{L^{2}}^{2}+C(\epsilon)
\|\nabla\theta\|_{L^2}^{\frac{2(2\alpha-1)}{4\alpha-1}}(1+\|\Lambda_{x_{1}}^{\alpha}\nabla\theta\|_{L^2}^{ 2} ) \,\|\Delta \theta\|_{L^2}^{2},\nonumber
\end{align}
\begin{align}
\mathcal{T}_{5} 
 \leq& C\|\Delta \theta\|_{L^2} \, \|\partial_{x_{1}}u_{2}\|_{L^2}^{1-\frac{1}{2\alpha}}\|\Lambda_{x_{1}}^{\alpha}\partial_{x_{1}}u_{2}\|_{L^2}^{ \frac{1}{2\alpha}} \,\|\partial_{x_{1}x_{2}}\theta\|_{L^2}^{1-\frac{1}{2\alpha}}
\|\Lambda_{x_{2}}^{\alpha}\partial_{x_{1}x_{2}}\theta\|_{L^2}^{ \frac{1}{2\alpha}}\nonumber\\
 \leq& C\|\Delta \theta\|_{L^2} \, \|\nabla\theta\|_{L^2}^{1-\frac{1}{2\alpha}}\|\Lambda_{x_{1}}^{\alpha}\nabla\theta\|_{L^2}^{ \frac{1}{2\alpha}} \,\|\Delta \theta\|_{L^2}^{1-\frac{1}{2\alpha}}\|\Lambda_{x_{1}}^{\alpha}\Delta \theta\|_{L^2}^{ \frac{1}{2\alpha}}\quad \Big(\mbox{using} \ (\ref{tuset008})\Big)
\nonumber\\
 \leq&
\epsilon\|\Lambda_{x_{1}}^{\alpha}\Delta\theta\|_{L^{2}}^{2}+C(\epsilon)
\|\nabla\theta\|_{L^2}^{\frac{2(2\alpha-1)}{4\alpha-1}}(1+\|\Lambda_{x_{1}}^{\alpha}\nabla\theta\|_{L^2}^{ 2} ) \,\|\Delta \theta\|_{L^2}^{2},\nonumber
\end{align}
\begin{align}
\mathcal{T}_{6} 
 \leq& C\|\Delta \theta\|_{L^2} \, \|\partial_{x_{2}}u_{2}\|_{L^2}^{1-\frac{1}{2\alpha}}\|\Lambda_{x_{2}}^{\alpha}\partial_{x_{2}}u_{2}\|_{L^2}^{ \frac{1}{2\alpha}} \,\|\partial_{x_{2}x_{2}}\theta\|_{L^2}^{1-\frac{1}{2\alpha}}
\|\Lambda_{x_{1}}^{\alpha}\partial_{x_{2}x_{2}}\theta\|_{L^2}^{ \frac{1}{2\alpha}}\nonumber\\
 \leq& C\|\Delta \theta\|_{L^2} \, \|\nabla\theta\|_{L^2}^{1-\frac{1}{2\alpha}}\|\Lambda_{x_{1}}^{\alpha}\nabla\theta\|_{L^2}^{ \frac{1}{2\alpha}} \,\|\Delta \theta\|_{L^2}^{1-\frac{1}{2\alpha}}\|\Lambda_{x_{1}}^{\alpha}\Delta \theta\|_{L^2}^{ \frac{1}{2\alpha}}\quad \Big(\mbox{using} \ (\ref{tuset010})\Big)
\nonumber\\
 \leq&
\epsilon\|\Lambda_{x_{1}}^{\alpha}\Delta\theta\|_{L^{2}}^{2}+C(\epsilon)
\|\nabla\theta\|_{L^2}^{\frac{2(2\alpha-1)}{4\alpha-1}}(1+\|\Lambda_{x_{1}}^{\alpha}\nabla\theta\|_{L^2}^{ 2} ) \,\|\Delta \theta\|_{L^2}^{2}.\nonumber
\end{align}
For the \textbf{Case 2}, one may conclude by using the inequality (\ref{qtri}) that
 \begin{align}
\mathcal{T}_{1} 
 \leq& C\|\Delta \theta\|_{L^2} \, \|\partial_{x_{1}}\theta\|_{L^2}^{1-\frac{1}{2\beta}}\|\Lambda_{x_{2}}^{\beta}\partial_{x_{1}}\theta\|_{L^2}^{ \frac{1}{2\beta}} \,\|\Delta u_{1}\|_{L^2}^{1-\frac{1}{2\beta}}\|\Lambda_{x_{1}}^{\beta}\Delta u_{1}\|_{L^2}^{ \frac{1}{2\beta}}\nonumber\\
 \leq& C\|\Delta \theta\|_{L^2} \, \|\nabla\theta\|_{L^2}^{1-\frac{1}{2\beta}}\|\Lambda_{x_{2}}^{\beta}\nabla\theta\|_{L^2}^{ \frac{1}{2\beta}} \,\|\Delta \theta\|_{L^2}^{1-\frac{1}{2\beta}}\|\Lambda_{x_{2}}^{\beta}\Delta \theta\|_{L^2}^{ \frac{1}{2\beta}}
\quad \Big(\mbox{using} \ (\ref{tuset002})\Big)\nonumber\\
 \leq&
\epsilon\|\Lambda_{x_{1}}^{\alpha}\Delta\theta\|_{L^{2}}^{2}+C(\epsilon)
\|\nabla\theta\|_{L^2}^{\frac{2(2\beta-1)}{4\beta-1}}(1+\|\Lambda_{x_{2}}^{\beta}\nabla\theta\|_{L^2}^{ 2} ) \,\|\Delta \theta\|_{L^2}^{2},\nonumber
\end{align}
\begin{align}
\mathcal{T}_{2}  
 \leq& C\|\Delta \theta\|_{L^2} \, \|\partial_{x_{2}}\theta\|_{L^2}^{1-\frac{1}{2\beta}}\|\Lambda_{x_{1}}^{\beta}\partial_{x_{2}}\theta\|_{L^2}^{ \frac{1}{2\beta}} \,\|\Delta u_{2}\|_{L^2}^{1-\frac{1}{2\beta}}\|\Lambda_{x_{2}}^{\beta}\Delta u_{2}\|_{L^2}^{ \frac{1}{2\beta}}\nonumber\\
 \leq& C\|\Delta \theta\|_{L^2} \, \|\nabla\theta\|_{L^2}^{1-\frac{1}{2\beta}}\|\Lambda_{x_{2}}^{\beta}\nabla\theta\|_{L^2}^{ \frac{1}{2\beta}} \,\|\Delta \theta\|_{L^2}^{1-\frac{1}{2\beta}}\|\Lambda_{x_{2}}^{\beta}\Delta \theta\|_{L^2}^{ \frac{1}{2\beta}}
\nonumber\\
 \leq&
\epsilon\|\Lambda_{x_{1}}^{\alpha}\Delta\theta\|_{L^{2}}^{2}+C(\epsilon)
\|\nabla\theta\|_{L^2}^{\frac{2(2\beta-1)}{4\beta-1}}(1+\|\Lambda_{x_{2}}^{\beta}\nabla\theta\|_{L^2}^{ 2} ) \,\|\Delta \theta\|_{L^2}^{2},\nonumber
\end{align}
\begin{align}
\mathcal{T}_{3}  
 \leq& C\|\Delta \theta\|_{L^2} \, \|\partial_{x_{1}}u_{1}\|_{L^2}^{1-\frac{1}{2\beta}}\|\Lambda_{x_{1}}^{\beta}\partial_{x_{1}}u_{1}\|_{L^2}^{ \frac{1}{2\beta}} \,\|\partial_{x_{1}x_{1}}\theta\|_{L^2}^{1-\frac{1}{2\beta}}
\|\Lambda_{x_{2}}^{\beta}\partial_{x_{1}x_{1}}\theta\|_{L^2}^{ \frac{1}{2\beta}}\nonumber\\
 \leq& C\|\Delta \theta\|_{L^2} \, \|\nabla\theta\|_{L^2}^{1-\frac{1}{2\beta}}\|\Lambda_{x_{2}}^{\beta}\nabla\theta\|_{L^2}^{ \frac{1}{2\beta}} \,\|\Delta \theta\|_{L^2}^{1-\frac{1}{2\beta}}\|\Lambda_{x_{2}}^{\beta}\Delta \theta\|_{L^2}^{ \frac{1}{2\beta}}\quad \Big(\mbox{using} \ (\ref{tuset005})\Big)
\nonumber\\
 \leq&
\epsilon\|\Lambda_{x_{1}}^{\beta}\Delta\theta\|_{L^{2}}^{2}+C(\epsilon)
\|\nabla\theta\|_{L^2}^{\frac{2(2\beta-1)}{4\beta-1}}(1+\|\Lambda_{x_{2}}^{\beta}\nabla\theta\|_{L^2}^{ 2} ) \,\|\Delta \theta\|_{L^2}^{2},\nonumber
\end{align}
\begin{align}
\mathcal{T}_{4}  
 \leq& C\|\Delta \theta\|_{L^2} \, \|\partial_{x_{2}}u_{1}\|_{L^2}^{1-\frac{1}{2\beta}}\|\Lambda_{x_{1}}^{\beta}\partial_{x_{2}}u_{1}\|_{L^2}^{ \frac{1}{2\beta}} \,\|\partial_{x_{1}x_{2}}\theta\|_{L^2}^{1-\frac{1}{2\beta}}
\|\Lambda_{x_{2}}^{\beta}\partial_{x_{1}x_{2}}\theta\|_{L^2}^{ \frac{1}{2\beta}}\nonumber\\
 \leq& C\|\Delta \theta\|_{L^2} \, \|\nabla\theta\|_{L^2}^{1-\frac{1}{2\beta}}\|\Lambda_{x_{2}}^{\beta}\nabla\theta\|_{L^2}^{ \frac{1}{2\beta}} \,\|\Delta \theta\|_{L^2}^{1-\frac{1}{2\beta}}\|\Lambda_{x_{2}}^{\beta}\Delta \theta\|_{L^2}^{ \frac{1}{2\beta}}\quad \Big(\mbox{using} \ (\ref{tuset007})\Big)
\nonumber\\
 \leq&
\epsilon\|\Lambda_{x_{1}}^{\beta}\Delta\theta\|_{L^{2}}^{2}+C(\epsilon)
\|\nabla\theta\|_{L^2}^{\frac{2(2\beta-1)}{4\beta-1}}(1+\|\Lambda_{x_{2}}^{\beta}\nabla\theta\|_{L^2}^{ 2} ) \,\|\Delta \theta\|_{L^2}^{2},\nonumber
\end{align}
\begin{align}
\mathcal{T}_{5} 
 \leq& C\|\Delta \theta\|_{L^2} \, \|\partial_{x_{1}}u_{2}\|_{L^2}^{1-\frac{1}{2\beta}}\|\Lambda_{x_{2}}^{\beta}\partial_{x_{1}}u_{2}\|_{L^2}^{ \frac{1}{2\beta}} \,\|\partial_{x_{1}x_{2}}\theta\|_{L^2}^{1-\frac{1}{2\beta}}
\|\Lambda_{x_{1}}^{\beta}\partial_{x_{1}x_{2}}\theta\|_{L^2}^{ \frac{1}{2\beta}}\nonumber\\
 \leq& C\|\Delta \theta\|_{L^2} \, \|\nabla\theta\|_{L^2}^{1-\frac{1}{2\beta}}\|\Lambda_{x_{2}}^{\beta}\nabla\theta\|_{L^2}^{ \frac{1}{2\beta}} \,\|\Delta \theta\|_{L^2}^{1-\frac{1}{2\beta}}\|\Lambda_{x_{2}}^{\beta}\Delta \theta\|_{L^2}^{ \frac{1}{2\beta}}\quad \Big(\mbox{using} \ (\ref{tuset009})\Big)
\nonumber\\
 \leq&
\epsilon\|\Lambda_{x_{1}}^{\beta}\Delta\theta\|_{L^{2}}^{2}+C(\epsilon)
\|\nabla\theta\|_{L^2}^{\frac{2(2\beta-1)}{4\beta-1}}(1+\|\Lambda_{x_{2}}^{\beta}\nabla\theta\|_{L^2}^{ 2} ) \,\|\Delta \theta\|_{L^2}^{2},\nonumber
\end{align}
\begin{align}
\mathcal{T}_{6}  
 \leq& C\|\Delta \theta\|_{L^2} \, \|\partial_{x_{2}}u_{2}\|_{L^2}^{1-\frac{1}{2\beta}}
 \|\Lambda_{x_{2}}^{\beta}\partial_{x_{2}}u_{2}\|_{L^2}^{ \frac{1}{2\beta}} \,\|\partial_{x_{2}x_{2}}\theta\|_{L^2}^{1-\frac{1}{2\beta}}
\|\Lambda_{x_{1}}^{\beta}\partial_{x_{2}x_{2}}\theta\|_{L^2}^{ \frac{1}{2\beta}}\nonumber\\
 \leq& C\|\Delta \theta\|_{L^2} \, \|\nabla\theta\|_{L^2}^{1-\frac{1}{2\beta}}\|\Lambda_{x_{2}}^{\beta}\nabla\theta\|_{L^2}^{ \frac{1}{2\beta}} \,\|\Delta \theta\|_{L^2}^{1-\frac{1}{2\beta}}\|\Lambda_{x_{2}}^{\beta}\Delta \theta\|_{L^2}^{ \frac{1}{2\beta}}\quad \Big(\mbox{using} \ (\ref{tuset011})\Big)
\nonumber\\
 \leq&
\epsilon\|\Lambda_{x_{1}}^{\beta}\Delta\theta\|_{L^{2}}^{2}+C(\epsilon)
\|\nabla\theta\|_{L^2}^{\frac{2(2\beta-1)}{4\beta-1}}(1+\|\Lambda_{x_{2}}^{\beta}\nabla\theta\|_{L^2}^{ 2} ) \,\|\Delta \theta\|_{L^2}^{2}.\nonumber
\end{align}
Combining the above estimates and taking $\epsilon$ suitable small, it allows us to get
\begin{eqnarray}
 \frac{d}{dt}\|\Delta\theta(t)\|_{L^{2}}^{2}+\|\Lambda_{x_{1}}^{\alpha}\Delta\theta\|_{L^{2}}^{2}+
 \|\Lambda_{x_{2}}^{\beta}\Delta\theta\|_{L^{2}}^{2}
\leq C
\|\nabla\theta\|_{L^2}^{\frac{2(2\alpha-1)}{4\alpha-1}}(1+\|\Lambda_{x_{1}}^{\alpha}\nabla\theta\|_{L^2}^{ 2} ) \,\|\Delta \theta\|_{L^2}^{2},\nonumber
\end{eqnarray}
for $\alpha>\frac{1}{2}$, while for $\beta>\frac{1}{2}$, one deduces
\begin{eqnarray}
 \frac{d}{dt}\|\Delta\theta(t)\|_{L^{2}}^{2}+\|\Lambda_{x_{1}}^{\alpha}\Delta\theta\|_{L^{2}}^{2}+
 \|\Lambda_{x_{2}}^{\beta}\Delta\theta\|_{L^{2}}^{2}
\leq C
\|\nabla\theta\|_{L^2}^{\frac{2(2\beta-1)}{4\beta-1}}(1+\|\Lambda_{x_{2}}^{\beta}\nabla\theta\|_{L^2}^{ 2} ) \,\|\Delta \theta\|_{L^2}^{2}.\nonumber
\end{eqnarray}
Applying the classical Gronwall inequality and noticing the key bounds (\ref{asddfgdfht302}) as well as (\ref{asd556fgdfsd}), we immediately conclude
$$\|\Delta\theta(t)\|_{L^{2}}^{2}+ \int_{0}^{t}{(\|\Lambda_{x_{1}}^{\alpha}\Delta\theta(
\tau)\|_{L^{2}}^{2}+\|\Lambda_{x_{2}}^{\beta}\Delta\theta(
\tau)\|_{L^{2}}^{2})\,d\tau}\leq
C(t,\,\theta_{0}).$$
Therefore, the proof of Proposition \ref{Lp5t12} is concluded.
\end{proof}

\vskip .1in
With the global
$H^2$-bound of $\theta$ in hand, we are now ready to establish the global $H^s$-estimate of $\theta$ to complete the proof of Theorem \ref{ThSQG}.
\begin{proof}[Proof of Theorem \ref{ThSQG}]
First, we need the following anisotropic interpolation inequality, whose proof will be provided in the appendix
\begin{eqnarray} \label{dpinty}\|h\|_{L^{\infty}}\leq C\|h\|_{L^{2}}^{1-
\frac{1}{2\delta_{1}}-\frac{1}{2\delta_{2}}}
\|\Lambda_{x_{1}}^{\delta_{1}}h\|_{L^{2}}^{\frac{1}{2\delta_{1}}}
\|\Lambda_{x_{2}}^{\delta_{2}}h\|_{L^{2}}^{\frac{1}{2\delta_{2}}},
\end{eqnarray}
where $\delta_{1}>0$ and $\delta_{2}>0$ satisfy $\frac{1}{\delta_{1}}+\frac{1}{\delta_{2}}<2$.
The above inequality further allows us to show that
\begin{align} \label{yrtew1}
\|\nabla \theta\|_{L^{\infty}} \leq&
C\|\nabla \theta\|_{L^{2}}^{1-
\frac{1}{2(1+\alpha)}-\frac{1}{2(1+\beta)}}
\|\Lambda_{x_{1}}^{1+\alpha}\nabla \theta\|_{L^{2}}^{\frac{1}{2(1+\alpha)}}
\|\Lambda_{x_{2}}^{1+\beta}\nabla \theta\|_{L^{2}}^{\frac{1}{2(1+\beta)}}
\nonumber\\ \leq&
C\|\nabla \theta\|_{L^{2}}^{1-
\frac{1}{2(1+\alpha)}-\frac{1}{2(1+\beta)}}
\|\Lambda_{x_{1}}^{\alpha}\Delta \theta\|_{L^{2}}^{\frac{1}{2(1+\alpha)}}
\|\Lambda_{x_{2}}^{\beta}\Delta \theta\|_{L^{2}}^{\frac{1}{2(1+\beta)}}
\end{align}
and
\begin{align} \label{yrtew2}
\|\nabla u\|_{L^{\infty}}
 \leq& C
\|\nabla u\|_{L^{2}}^{1-
\frac{1}{2(1+\alpha)}-\frac{1}{2(1+\beta)}}
\|\Lambda_{x_{1}}^{1+\alpha}\nabla u\|_{L^{2}}^{\frac{1}{2(1+\alpha)}}
\|\Lambda_{x_{2}}^{1+\beta}\nabla u\|_{L^{2}}^{\frac{1}{2(1+\beta)}}
\nonumber\\ \leq& C
\|\nabla \theta\|_{L^{2}}^{1-
\frac{1}{2(1+\alpha)}-\frac{1}{2(1+\beta)}}
\|\Lambda_{x_{1}}^{1+\alpha}\nabla \theta\|_{L^{2}}^{\frac{1}{2(1+\alpha)}}
\|\Lambda_{x_{2}}^{1+\beta}\nabla \theta\|_{L^{2}}^{\frac{1}{2(1+\beta)}}
\nonumber\\ \leq&
C\|\nabla \theta\|_{L^{2}}^{1-
\frac{1}{2(1+\alpha)}-\frac{1}{2(1+\beta)}}
\|\Lambda_{x_{1}}^{\alpha}\Delta \theta\|_{L^{2}}^{\frac{1}{2(1+\alpha)}}
\|\Lambda_{x_{2}}^{\beta}\Delta \theta\|_{L^{2}}^{\frac{1}{2(1+\beta)}}.
\end{align}
The obtained estimates in (\ref{asddfgdfht302}), (\ref{asd556fgdfsd}) and (\ref{t3326t017}) yield
\begin{eqnarray*}  \int_{0}^{t}(\|\nabla \theta(\tau)\|_{L^{\infty}}^{\frac{4(1+\alpha)(1+\beta)}{2+\alpha+\beta}}+\|\nabla u(\tau)\|_{L^{\infty}}^{\frac{4(1+\alpha)(1+\beta)}{2+\alpha+\beta}})\,d\tau\leq
C(t,\,\theta_{0}).
\end{eqnarray*}
The basic $H^{s}$-estimate of the system (\ref{SQG}) reads
\begin{eqnarray*}
 \frac{d}{dt}
\|\theta(t)\|_{H^{s}}^{2} +\|\Lambda_{x_{1}}^{\alpha}\theta\|_{H^{s}}^{2}
+\|\Lambda_{x_{2}}^{\beta}\theta\|_{H^{s}}^{2}
 \leq  C(1+\|\nabla u\|_{L^{\infty}}+\|\nabla \theta\|_{L^{\infty}}) \|\theta\|_{H^{s}}^{2}.
\end{eqnarray*}
It is then clear that
$$
\|\theta(t)\|_{H^{s}}^{2}+\int_{0}^{t}(\|\Lambda_{x_{1}}^{\alpha}\theta(\tau)\|_{H^{s}}^{2}
+\|\Lambda_{x_{2}}^{\beta}\theta(\tau)\|_{H^{s}}^{2})\,d\tau\leq
C(t,\,\theta_{0}).$$

Finally, we are going to show the uniqueness. In fact, we can prove the uniqueness result in space ${H}^{1}$, namely,
$$\mathcal{Z}:=\left\{\theta: \ \|\nabla\theta(t)\|_{L_{T}^{\infty}L^{2}}^{2}+ \int_{0}^{T}{(\|\Lambda_{x_{1}}^{\alpha}\nabla\theta(
\tau)\|_{L^{2}}^{2}+\|\Lambda_{x_{2}}^{\beta}\nabla\theta(
\tau)\|_{L^{2}}^{2})\,d\tau}<\infty\right\},$$
where $\alpha\in (0,\,1)$ and $\beta\in (0,\,1)$ satisfy \eqref{sdf2334}.
Obviously, we remark that the uniqueness holds true in $H^{s}$ for any $s>1$.
To this end, we consider two solutions $\theta^{(1)}$ and $\theta^{(2)}$ of \eqref{SQG}, emanating from the same initial data, and belonging to $\mathcal{Z}$. We denote $\widetilde{\theta}=\theta^{(1)}-\theta^{(2)}$ and $\widetilde{u}=u^{(1)}-u^{(2)}$, where $\widetilde{u}_{1}=-\mathcal
{R}_{2}\widetilde{\theta}$ and $\widetilde{u}_{2}=\mathcal {R}_{1}\widetilde{\theta}$. Then, we get
\begin{equation}\label{diffSQG}
\left\{\aligned
&\partial_{t}\widetilde{\theta}+(u^{(1)}\cdot \nabla)\widetilde{\theta}+\Lambda_{x_{1}}^{2\alpha}\widetilde{\theta}+
\Lambda_{x_{2}}^{2\beta}\widetilde{\theta}=-(\widetilde{u} \cdot \nabla)\theta^{(2)},\\
&\widetilde{\theta}(x, 0)=0.
\endaligned\right.
\end{equation}
Applying the basic $L^{2}$-estimate to \eqref{diffSQG} yields
\begin{align}
\frac{1}{2}\frac{d}{dt}\|\widetilde{\theta}(t)\|_{L^{2}}^{2}+
 \|\Lambda_{x_{1}}^{\alpha}\widetilde{\theta}\|_{L^{2}}^{2}+
 \|\Lambda_{x_{2}}^{\beta}\widetilde{\theta}\|_{L^{2}}^{2}
 &=-\int_{\mathbb{R}^{2}}{(\widetilde{u} \cdot \nabla)\theta^{(2)}\widetilde{\theta}\,dx}\nonumber\\
 &=-\int_{\mathbb{R}^{2}}{\widetilde{u}_{1}\partial_{x_{1}}\theta^{(2)}
 \widetilde{\theta}\,dx}-\int_{\mathbb{R}^{2}}{\widetilde{u}_{2}\partial_{x_{2}}
 \theta^{(2)}\widetilde{\theta}\,dx}.\nonumber
 \end{align}
By the same argument adopted in the proof of Proposition \ref{Lp5t12},
we know that if $\alpha$ and $\beta$ satisfy (\ref{sdf2334}), then $\alpha>\frac{1}{2}$ or $\beta>\frac{1}{2}$ holds true. For the case $\alpha>\frac{1}{2}$, we deduce by using (\ref{qtri}), $\widetilde{u}_{1}=-\mathcal
{R}_{2}\widetilde{\theta}$ and $\widetilde{u}_{2}=\mathcal {R}_{1}\widetilde{\theta}$ that
\begin{align}
\left|-\int_{\mathbb{R}^{2}}{\widetilde{u}_{1}\partial_{x_{1}}\theta^{(2)}
 \widetilde{\theta}\,dx}\right|&\leq C\|\widetilde{\theta}\|_{L^{2}}\|\partial_{x_{1}}\theta^{(2)}\|_{L^{2}}
 ^{1-\frac{1}{2\alpha}}
 \|\Lambda_{x_{2}}^{\alpha}\partial_{x_{1}}\theta^{(2)}\|_{L^{2}}^{\frac{1}{2\alpha}}
 \|\widetilde{u}_{1}\|_{L^{2}}^{1-\frac{1}{2\alpha}}
 \|\Lambda_{x_{1}}^{\alpha}\widetilde{u}_{1}\|_{L^{2}}^{\frac{1}{2\alpha}}\nonumber\\
 &\leq C\|\widetilde{\theta}\|_{L^{2}}\|\nabla\theta^{(2)}\|_{L^{2}}
 ^{1-\frac{1}{2\alpha}}
 \|\Lambda_{x_{1}}^{\alpha}\nabla\theta^{(2)}\|_{L^{2}}^{\frac{1}{2\alpha}}
 \|\widetilde{\theta}\|_{L^{2}}^{1-\frac{1}{2\alpha}}
 \|\Lambda_{x_{1}}^{\alpha}\widetilde{\theta}\|_{L^{2}}^{\frac{1}{2\alpha}}
 \nonumber\\
 &\leq
 \frac{1}{8}\|\Lambda_{x_{1}}^{\alpha}\widetilde{\theta}\|_{L^{2}}^{2}+
 C\|\nabla\theta^{(2)}\|_{L^{2}}^{\frac{2(2\alpha-1)}{4\alpha-1}} (1+\|\Lambda_{x_{1}}^{\alpha}\nabla\theta^{(2)}\|_{L^{2}}^{2})
 \|\widetilde{\theta}\|_{L^{2}}^{2},\nonumber
 \end{align}
 \begin{align}
\left|-\int_{\mathbb{R}^{2}}{\widetilde{u}_{2}\partial_{x_{2}}\theta^{(2)}
 \widetilde{\theta}\,dx}\right|&\leq C\|\widetilde{\theta}\|_{L^{2}}\|\partial_{x_{2}}\theta^{(2)}\|_{L^{2}}
 ^{1-\frac{1}{2\alpha}}
 \|\Lambda_{x_{1}}^{\alpha}\partial_{x_{2}}\theta^{(2)}\|_{L^{2}}^{\frac{1}{2\alpha}}
 \|\widetilde{u}_{2}\|_{L^{2}}^{1-\frac{1}{2\alpha}}
 \|\Lambda_{x_{2}}^{\alpha}\widetilde{u}_{2}\|_{L^{2}}^{\frac{1}{2\alpha}}\nonumber\\
 &\leq C\|\widetilde{\theta}\|_{L^{2}}\|\nabla\theta^{(2)}\|_{L^{2}}
 ^{1-\frac{1}{2\alpha}}
 \|\Lambda_{x_{1}}^{\alpha}\nabla\theta^{(2)}\|_{L^{2}}^{\frac{1}{2\alpha}}
 \|\widetilde{\theta}\|_{L^{2}}^{1-\frac{1}{2\alpha}}
 \|\Lambda_{x_{1}}^{\alpha}\widetilde{\theta}\|_{L^{2}}^{\frac{1}{2\alpha}}
 \nonumber\\
 &\leq
 \frac{1}{8}\|\Lambda_{x_{1}}^{\alpha}\widetilde{\theta}\|_{L^{2}}^{2}+
 C\|\nabla\theta^{(2)}\|_{L^{2}}^{\frac{2(2\alpha-1)}{4\alpha-1}} (1+\|\Lambda_{x_{1}}^{\alpha}\nabla\theta^{(2)}\|_{L^{2}}^{2})
 \|\widetilde{\theta}\|_{L^{2}}^{2},\nonumber
 \end{align}
where we have used the fact due to Plancherel's Theorem and $\widetilde{u}_{2}=\mathcal {R}_{1}\widetilde{\theta}$
$$\|\Lambda_{x_{2}}^{\alpha}\widetilde{u}_{2}\|_{L^{2}}\leq C\|\Lambda_{x_{1}}^{\alpha}\widetilde{\theta}\|_{L^{2}}.$$
Similarly, for the case $\beta>\frac{1}{2}$, it implies
\begin{align}
\left|-\int_{\mathbb{R}^{2}}{\widetilde{u}_{1}\partial_{x_{1}}\theta^{(2)}
 \widetilde{\theta}\,dx}\right|&\leq C\|\widetilde{\theta}\|_{L^{2}}\|\partial_{x_{1}}\theta^{(2)}\|_{L^{2}}
 ^{1-\frac{1}{2\beta}}
 \|\Lambda_{x_{2}}^{\beta}\partial_{x_{1}}\theta^{(2)}\|_{L^{2}}^{\frac{1}{2\beta}}
 \|\widetilde{u}_{1}\|_{L^{2}}^{1-\frac{1}{2\beta}}
 \|\Lambda_{x_{1}}^{\beta}\widetilde{u}_{1}\|_{L^{2}}^{\frac{1}{2\beta}}\nonumber\\
 &\leq C\|\widetilde{\theta}\|_{L^{2}}\|\nabla\theta^{(2)}\|_{L^{2}}
 ^{1-\frac{1}{2\beta}}
 \|\Lambda_{x_{2}}^{\beta}\nabla\theta^{(2)}\|_{L^{2}}^{\frac{1}{2\beta}}
 \|\widetilde{\theta}\|_{L^{2}}^{1-\frac{1}{2\beta}}
 \|\Lambda_{x_{2}}^{\beta}\widetilde{\theta}\|_{L^{2}}^{\frac{1}{2\beta}}
 \nonumber\\
 &\leq
 \frac{1}{8}\|\Lambda_{x_{2}}^{\beta}\widetilde{\theta}\|_{L^{2}}^{2}+
 C\|\nabla\theta^{(2)}\|_{L^{2}}^{\frac{2(2\beta-1)}{4\beta-1}} (1+\|\Lambda_{x_{2}}^{\beta}\nabla\theta^{(2)}\|_{L^{2}}^{2})
 \|\widetilde{\theta}\|_{L^{2}}^{2},\nonumber
 \end{align}
\begin{align}
\left|-\int_{\mathbb{R}^{2}}{\widetilde{u}_{2}\partial_{x_{2}}\theta^{(2)}
 \widetilde{\theta}\,dx}\right|&\leq C\|\widetilde{\theta}\|_{L^{2}}\|\partial_{x_{2}}\theta^{(2)}\|_{L^{2}}
 ^{1-\frac{1}{2\beta}}
 \|\Lambda_{x_{1}}^{\beta}\partial_{x_{2}}\theta^{(2)}\|_{L^{2}}^{\frac{1}{2\beta}}
 \|\widetilde{u}_{2}\|_{L^{2}}^{1-\frac{1}{2\beta}}
 \|\Lambda_{x_{2}}^{\beta}\widetilde{u}_{2}\|_{L^{2}}^{\frac{1}{2\beta}}\nonumber\\
 &\leq C\|\widetilde{\theta}\|_{L^{2}}\|\nabla\theta^{(2)}\|_{L^{2}}
 ^{1-\frac{1}{2\beta}}
 \|\Lambda_{x_{2}}^{\beta}\nabla\theta^{(2)}\|_{L^{2}}^{\frac{1}{2\beta}}
 \|\widetilde{\theta}\|_{L^{2}}^{1-\frac{1}{2\beta}}
 \|\Lambda_{x_{2}}^{\beta}\widetilde{\theta}\|_{L^{2}}^{\frac{1}{2\beta}}
 \nonumber\\
 &\leq
 \frac{1}{8}\|\Lambda_{x_{2}}^{\beta}\widetilde{\theta}\|_{L^{2}}^{2}+
 C\|\nabla\theta^{(2)}\|_{L^{2}}^{\frac{2(2\beta-1)}{4\beta-1}} (1+\|\Lambda_{x_{2}}^{\beta}\nabla\theta^{(2)}\|_{L^{2}}^{2})
 \|\widetilde{\theta}\|_{L^{2}}^{2}.\nonumber
 \end{align}
Therefore, we obtain
\begin{align}&\label{uniqc341}
\frac{d}{dt}\|\widetilde{\theta}(t)\|_{L^{2}}^{2}+
 \|\Lambda_{x_{1}}^{\alpha}\widetilde{\theta}\|_{L^{2}}^{2}+
 \|\Lambda_{x_{2}}^{\beta}\widetilde{\theta}\|_{L^{2}}^{2}\nonumber\\
 &\leq C\|\nabla\theta^{(2)}\|_{L^{2}}^{\frac{2(2\alpha-1)}
 {4\alpha-1}}  (1+\|\Lambda_{x_{1}}^{\alpha}\nabla\theta^{(2)}\|_{L^{2}}^{2})
 \|\widetilde{\theta}\|_{L^{2}}^{2},\quad \mbox{for}\ \ \alpha>\frac{1}{2};
 \end{align}
 \begin{align}&\label{uniqc342}
\frac{d}{dt}\|\widetilde{\theta}(t)\|_{L^{2}}^{2}+
 \|\Lambda_{x_{1}}^{\alpha}\widetilde{\theta}\|_{L^{2}}^{2}+
 \|\Lambda_{x_{2}}^{\beta}\widetilde{\theta}\|_{L^{2}}^{2}\nonumber\\
 &\leq C\|\nabla\theta^{(2)}\|_{L^{2}}^{\frac{2(2\beta-1)}{4\beta-1}} (1
 +\|\Lambda_{x_{2}}^{\beta}\nabla\theta^{(2)}\|_{L^{2}}^{2})
 \|\widetilde{\theta}\|_{L^{2}}^{2},\quad \mbox{for}\ \ \beta>\frac{1}{2}.
 \end{align}
The above estimates \eqref{uniqc341} and \eqref{uniqc342} along with the Gronwall inequality give
$$\widetilde{\theta}(t)=0\quad \mbox{on}\ [0,\,T].$$
This yields the uniqueness of the solution on $[0,\,T]$.
Consequently, we complete the proof of Theorem \ref{ThSQG}.
\end{proof}

\vskip .2in
\appendix

\section{Besov spaces and several inequalities}\label{appd11}

In this section, we show some common notations about the Besov spaces and several inequalities.
Now let us begin with the Littlewood-Paley theory (see for instance \cite{BCD}).
We choose
some smooth radial non increasing function $\chi$ with values in $[0, 1]$ such that $\chi\in
C_{0}^{\infty}(\mathbb{R}^{n})$ is supported in the ball
$\mathcal{B}:=\{\xi\in \mathbb{R}^{n}, |\xi|\leq \frac{4}{3}\}$ and
and with value $1$ on $\{\xi\in \mathbb{R}^{n}, |\xi|\leq \frac{3}{4}\}$, then we set
$\varphi(\xi)=\chi\big(\frac{\xi}{2}\big)-\chi(\xi)$. One easily verifies that
${\varphi\in C_{0}^{\infty}(\mathbb{R}^{n})}$ is supported in the annulus
$\mathcal{C}:=\{\xi\in \mathbb{R}^{n}, \frac{3}{4}\leq |\xi|\leq
\frac{8}{3}\}$ and satisfy
$$\chi(\xi)+\sum_{j\geq0}\varphi(2^{-j}\xi)=1, \quad  \forall \xi\in \mathbb{R}^{n}.$$
Let $h=\mathcal{F}^{-1}(\varphi)$ and $\widetilde{h}=\mathcal{F}^{-1}(\chi)$, then we introduce the dyadic blocks $\Delta_{j}$ of our decomposition by setting
$$\Delta_{j}u=0,\ \ j\leq -2; \ \  \ \ \ \Delta_{-1}u=\chi(D)u=\int_{\mathbb{R}^{n}}{\widetilde{h}(y)u(x-y)\,dy};
$$
$$ \Delta_{j}u=\varphi(2^{-j}D)u=2^{jn}\int_{\mathbb{R}^{n}}{h(2^{j}y)u(x-y)\,dy},\ \ \forall j\in \mathbb{N}.
$$
We shall also use the
following low-frequency cut-off:
$$\ S_{j}u=\chi(2^{-j}D)u=\sum_{-1\leq k\leq j-1} \Delta_{k}u=2^{jn}\int_{\mathbb{R}^{n}}{\widetilde{h}(2^{j}y)u(x-y)\,dy},\ \ \forall j\in \mathbb{N}.$$
Meanwhile, we define the homogeneous dyadic blocks as
$$\dot{\Delta}_{j}u=\varphi(2^{-j}D)u =2^{jn}\int_{\mathbb{R}^{n}}{h(2^{j}y)u(x-y)\,dy},\ \ \forall j\in \mathbb{Z}.$$
\vskip .1in

We denote the function spaces of rapidly decreasing functions by $S(\mathbb{R}^{n})$, tempered
distributions by $S'(\mathbb{R}^{n})$, and polynomials by $\mathcal{P}(\mathbb{R}^{n})$.
Let us now recall the definition of homogeneous and inhomogeneous Besov spaces through
the dyadic decomposition.
\begin{define}
Let $s\in \mathbb{R}, (p,r)\in[1,+\infty]^{2}$. The homogeneous
Besov space $\dot{B}_{p,r}^{s}$ is defined as a space of $f\in
S'(\mathbb{R}^{n})/\mathcal{P}(\mathbb{R}^{n})$ such that
$$ \dot{B}_{p,r}^{s}=\{f\in S'(\mathbb{R}^{n})/\mathcal{P}(\mathbb{R}^{n});  \|f\|_{\dot{B}_{p,r}^{s}}<\infty\},$$
where
\begin{equation}\nonumber
 \|f\|_{\dot{B}_{p,r}^{s}}=\left\{\aligned
&\Big(\sum_{j\in \mathbb{Z}}2^{jrs}\|\dot{\Delta}_{j}f\|_{L^{p}}^{r}\Big)^{\frac{1}{r}}, \quad \forall \ r<\infty,\\
&\sup_{j\in \mathbb{Z}}
2^{js}\|\dot{\Delta}_{j}f\|_{L^{p}}, \quad \forall \ r=\infty.\\
\endaligned\right.
\end{equation}
\end{define}

\begin{define}
Let $s\in \mathbb{R}, (p,r)\in[1,+\infty]^{2}$. The inhomogeneous
Besov space $B_{p,r}^{s}$ is defined as a space of $f\in
S'(\mathbb{R}^{n})$ such that
$$ B_{p,r}^{s}=\{f\in S'(\mathbb{R}^{n});  \|f\|_{B_{p,r}^{s}}<\infty\},$$
where
\begin{equation}\nonumber
 \|f\|_{B_{p,r}^{s}}=\left\{\aligned
&\Big(\sum_{j\geq-1}2^{jrs}\|\Delta_{j}f\|_{L^{p}}^{r}\Big)^{\frac{1}{r}}, \quad \forall \ r<\infty,\\
&\sup_{j\geq-1}
2^{js}\|\Delta_{j}f\|_{L^{p}}, \quad \forall \ r=\infty.\\
\endaligned\right.
\end{equation}
\end{define}

\vskip .1in
Next, we introduce the Bernstein lemma which is fundamental in the analysis involving Besov spaces.

\begin{lemma} [see \cite{BCD}]\label{lcfgyrw}
 Let $k\geq 0, 1\leq a\leq b\leq\infty$. Assume that
$$
\mbox{supp}\, \widehat{f} \subset \{\xi\in \mathbb{R}^n: \,\, |\xi|
\lesssim  2^j \},
$$
for some integer $j$,  then there exists a constant $C_1$ such that
$$
\|\Lambda^{k}f\|_{L^b} \le C_1\, 2^{j k  +
jn(\frac{1}{a}-\frac{1}{b})} \|f\|_{L^a}.
$$
If $f$ satisfies
\begin{equation*}
\mbox{supp}\, \widehat{f} \subset \{\xi\in \mathbb{R}^n: \,\,|\xi|
\thickapprox 2^j \}
\end{equation*}
for some integer $j$, then
$$
C_1\, 2^{ j k} \|f\|_{L^b } \le \|\Lambda^{k} f\|_{L^b } \le
C_2\, 2^{  j k + j n(\frac{1}{a}-\frac{1}{b})} \|f\|_{L^a},
$$
where $C_1$ and $C_2$ are constants depending on $\alpha,\,a$ and $b$
only.
\end{lemma}

\vskip .2in
\begin{center}{An alternative proof of (\ref{dpinty})}
\end{center}
Here we give the proof of the anisotropic interpolation inequality (\ref{dpinty}).
Before proving this inequality, we point out that the anisotropic interpolation inequality established in \cite[Lemma A.2]{DP3} is a direct consequence of the inequality (\ref{dpinty}).
By means of the following one-dimensional Sobolev inequality
\begin{eqnarray}
\|g\|_{L_{x_{1}}^{\infty}(\mathbb{R})}\leq C\|g\|_{L_{x_{1}}^{2}(\mathbb{R})}^{\frac{2\gamma-1}{2\gamma}} \|\Lambda_{x_{1}}^{\gamma}g\|_{L_{x_{1}}^{2}(\mathbb{R})}^{ \frac{1}{2\gamma}},\quad \gamma>\frac{1}{2},\nonumber
\end{eqnarray}
it is clear that by choosing the intermediate variables $\varepsilon_{1},\,\varepsilon_{2}>\frac{1}{2}$ and noticing $\delta_{2}>\frac{1}{2}$
\begin{align}
\|h(x_{1},x_{2})\|_{L^{\infty}}=&
\|h(x_{1},x_{2})\|_{L_{x_{2}}^{\infty}L_{x_{1}}^{\infty}}\nonumber\\ \leq& C\|h(x_{1},x_{2})\|_{L_{x_{2}}^{\infty}L_{x_{1}}^{2}}^{\frac{2\gamma-1}{2\varepsilon_{1}}} \|\Lambda_{x_{1}}^{\varepsilon_{1}}h(x_{1},x_{2})\|_{L_{x_{2}}^{\infty}L_{x_{1}}^{2}}^{ \frac{1}{2\varepsilon_{1}}}\nonumber\\
 \leq& C\|h(x_{1},x_{2})\|_{L_{x_{1}}^{2}L_{x_{2}}^{\infty}}^{\frac{2\varepsilon_{1}-1}{2\varepsilon_{1}}} \|\Lambda_{x_{1}}^{\varepsilon_{1}}h(x_{1},x_{2})\|_{L_{x_{1}}^{2}L_{x_{2}}^{\infty}}^{ \frac{1}{2\varepsilon_{1}}}\nonumber\\
 \leq& C\|h(x_{1},x_{2})\|_{L_{x_{1}}^{2}L_{x_{2}}^{2}}^{\frac{2\varepsilon_{1}-1}{2\varepsilon_{1}}
\frac{2\delta_{2}-1}{2\delta_{2}}}
\|\Lambda_{x_{2}}^{\delta_{2}}h(x_{1},x_{2})\|_{L_{x_{1}}^{2}L_{x_{2}}^{2}}^{\frac{2\varepsilon_{1}-1}{2\varepsilon_{1}}
\frac{1}{2\delta_{2}}}
\nonumber\\& \times
\|\Lambda_{x_{1}}^{\varepsilon_{1}}h(x_{1},x_{2})\|_{L_{x_{1}}^{2}L_{x_{2}}^{2}}^{ \frac{1}{2\varepsilon_{1}}\frac{2\varepsilon_{2}-1}{2\varepsilon_{2}}}
\|\Lambda_{x_{2}}^{\varepsilon_{2}}\Lambda_{x_{1}}^{\varepsilon_{1}}h(x_{1},x_{2})\|_{L_{x_{1}}^{2}L_{x_{2}}^{2}}^{ \frac{1}{2\varepsilon_{1}}\frac{1}{2\varepsilon_{2}}}
\nonumber\\
 =& C\|h(x_{1},x_{2})\|_{L^{2}}^{\frac{2\varepsilon_{1}-1}{2\varepsilon_{1}}
\frac{2\delta_{2}-1}{2\delta_{2}}}
\|\Lambda_{x_{2}}^{\delta_{2}}h(x_{1},x_{2})\|_{L^{2}}^{\frac{2\varepsilon_{1}-1}{2\varepsilon_{1}}
\frac{1}{2\delta_{2}}}
\nonumber\\& \times
\|\Lambda_{x_{1}}^{\varepsilon_{1}}h(x_{1},x_{2})\|_{L^{2}}^{ \frac{1}{2\varepsilon_{1}}\frac{2\varepsilon_{2}-1}{2\varepsilon_{2}}}
\|\Lambda_{x_{2}}^{\varepsilon_{2}}\Lambda_{x_{1}}^{\varepsilon_{1}}h(x_{1},x_{2})
\|_{L^{2}}^{ \frac{1}{2\varepsilon_{1}}\frac{1}{2\varepsilon_{2}}}.\nonumber
\end{align}
Now if we further assume $\varepsilon_{1}\leq \delta_{1},\,\varepsilon_{2}\leq \delta_{2}$ and $\frac{\varepsilon_{1}}{\delta_{1}}+\frac{\varepsilon_{2}}{\delta_{2}}\leq 1$, then we obtain
\begin{align}
\|\Lambda_{x_{1}}^{\varepsilon_{1}}h(x_{1},x_{2})\|_{L^{2}} \leq \|h(x_{1},x_{2})\|_{L^{2}}^{\frac{\delta_{1}-\varepsilon_{1}}{\delta_{1}}}
\|\Lambda_{x_{1}}^{\delta_{1}}h(x_{1},x_{2})\|_{L^{2}}^{\frac{\varepsilon_{1}}
{\delta_{1}}},\nonumber
\end{align}
and
\begin{align}
\|\Lambda_{x_{2}}^{\varepsilon_{2}}\Lambda_{x_{1}}^{\varepsilon_{1}}h(x_{1},x_{2})
\|_{L^{2}}
 =&\Big(\int_{\mathbb{R}^{2}}{|\xi_{2}|^{2\varepsilon_{2}}|\xi_{1}|^{2\varepsilon_{1}}
|\widehat{h}(\xi)|^{2}\,d\xi}\Big)^{\frac{1}{2}}\nonumber\\
 =&\Big(\int_{\mathbb{R}^{2}}{\big(|\xi_{2}|^{2\varepsilon_{2}}
|\widehat{h}(\xi)|^{\frac{2\varepsilon_{2}}{\delta_{2}}}\big)
\big(|\xi_{1}|^{2\varepsilon_{1}}|\widehat{h}(\xi)|
^{\frac{2\varepsilon_{1}}{\delta_{1}}}}\big)
|\widehat{h}(\xi)|^{2-\frac{2\varepsilon_{2}}
{\delta_{2}}-\frac{2\varepsilon_{1}}{\delta_{1}} }\,d\xi\Big)^{\frac{1}{2}}
\nonumber\\
 \leq&C\Big(\int_{\mathbb{R}^{2}}{|\xi_{2}|^{2\delta_{2}}
|\widehat{h}(\xi)|^{2}\,d\xi}\Big)^{\frac{\varepsilon_{2}}{2\delta_{2}}}
\Big(\int_{\mathbb{R}^{2}}{|\xi_{1}|^{2\delta_{1}}
|\widehat{h}(\xi)|^{2}\,d\xi}\Big)^{\frac{\varepsilon_{1}}{2\delta_{1}}}\nonumber\\&
\times
\Big(\int_{\mathbb{R}^{2}}{
|\widehat{h}(\xi)|^{2}\,d\xi}\Big)^{\frac{1}{2}-\frac{\varepsilon_{1}}{2\delta_{1}}
-\frac{\varepsilon_{2}}{2\delta_{2}}}
\nonumber\\
 =&C
\|\Lambda_{x_{2}}^{\delta_{2}}h(x_{1},x_{2})\|_{L^{2}}
^{\frac{\varepsilon_{2}}{\delta_{2}}}
\|\Lambda_{x_{1}}^{\delta_{1}}h(x_{1},x_{2})\|_{L^{2}}
^{\frac{\varepsilon_{1}}{\delta_{1}}}
\|h(x_{1},x_{2})\|_{L^{2}}^{1-\frac{\varepsilon_{1}}{\delta_{1}}
-\frac{\varepsilon_{2}}{\delta_{2}}}.\nonumber
\end{align}
Combining the above estimates, it yields
\begin{align}
\|h(x_{1},x_{2})\|_{L^{\infty}} \leq& C \|h(x_{1},x_{2})\|_{L^{2}}^{\frac{2\varepsilon_{1}-1}{2\varepsilon_{1}}
\frac{2\delta_{2}-1}{2\delta_{2}}+\frac{1}{2\varepsilon_{1}}
\frac{2\varepsilon_{2}-1}{2\varepsilon_{2}}
\frac{\delta_{1}-\varepsilon_{1}}{\delta_{1}}+(1-\frac{\varepsilon_{1}}{\delta_{1}}
-\frac{\varepsilon_{2}}{\delta_{2}})\frac{1}{2\varepsilon_{1}}
\frac{1}{2\varepsilon_{2}}}\nonumber\\& \times
\|\Lambda_{x_{1}}^{\delta_{1}}h(x_{1},x_{2})\|_{L^{2}}^{\frac{1}{2\varepsilon_{1}}
\frac{2\varepsilon_{2}-1}{2\varepsilon_{2}}\frac{\varepsilon_{1}}{\delta_{1}}+
\frac{\varepsilon_{1}}{\delta_{1}}\frac{1}{2\varepsilon_{1}}\frac{1}
{2\varepsilon_{2}}}
\|\Lambda_{x_{2}}^{\delta_{2}}h(x_{1},x_{2})\|_{L^{2}}^{\frac{2\varepsilon_{1}-1}{2\varepsilon_{1}}
\frac{1}{2\delta_{2}}+\frac{\varepsilon_{2}}{\delta_{2}}\frac{1}{2\varepsilon_{1}}\frac{1}{2\varepsilon_{2}}}
\nonumber\\
 =& C \|h(x_{1},x_{2})\|_{L^{2}}^{1-
\frac{1}{2\delta_{1}}-\frac{1}{2\delta_{2}}}
\|\Lambda_{x_{1}}^{\delta_{1}}h(x_{1},x_{2})\|_{L^{2}}^{\frac{1}{2\delta_{1}}}
\|\Lambda_{x_{2}}^{\delta_{2}}h(x_{1},x_{2})\|_{L^{2}}^{\frac{1}{2\delta_{2}}},
\nonumber
\end{align}
where the intermediate variables $\varepsilon_{1}$ and $\varepsilon_{2}$ should be satisfied $\frac{1}{2}<\varepsilon_{1}\leq \delta_{1},\,\frac{1}{2}<\varepsilon_{2}\leq \delta_{2}$ and $\frac{\varepsilon_{1}}{\delta_{1}}+\frac{\varepsilon_{2}}{\delta_{2}}\leq 1$. Thus, it leads to
$\frac{\varepsilon_{1}}{\delta_{1}}>\frac{1}{2\delta_{1}}$ and $\frac{\varepsilon_{2}}{\delta_{2}}>\frac{1}{2\delta_{2}}$, which together with the condition $\frac{\varepsilon_{1}}{\delta_{1}}+\frac{\varepsilon_{2}}{\delta_{2}}\leq 1$ implies
\begin{eqnarray}\label{atasdgffdag}\frac{1}{2\delta_{1}}+\frac{1}{2\delta_{2}}<1\quad\mbox{or}\quad \frac{1}{\delta_{1}}+\frac{1}{\delta_{2}}<2.\end{eqnarray}
The above argument implies that  the intermediate variables $\varepsilon_{1}$ and $\varepsilon_{2}$ do exist as long as (\ref{atasdgffdag}) holds true. This completes the proof of the inequality (\ref{dpinty}).

\vskip .3in

\section{Local well-posedness theory of (\ref{SQG})}\label{appd12}

For the sake of completeness, this appendix presents the local
existence and uniqueness result for (\ref{SQG}) with initial
data $\theta_{0}\in H^{s}(\mathbb{R}^{2})$ for $s\geq2$.
More precisely, in this appendix, we prove the following local well-posedness result.
\begin{Pros}\label{P1}
Let $\theta_{0}\in H^{s}(\mathbb{R}^{2})$ with $s\geq2$ and $\alpha,\,\beta>0$. Then there exists a positive time $T$ depending on $\|\theta_{0}\|_{H^{s}}$ such that (\ref{SQG}) admits a unique solution $\theta\in C([0, T];H^{s}(\mathbb{R}^{2}))$.
\end{Pros}

The proof of Proposition \ref{P1} can be performed by the method similar to Chapter 3 in \cite{Majdab}.
To prove Proposition \ref{P1}, the main step is to approximate  (\ref{SQG}) in order to easily produce a family of global
smooth solutions. In order to do this, we may  for instance make use
of the Friedrichs method. Now we define the spectral cut-off as
follows
$$\widehat{\mathcal {J}_{N}f}(\xi)=\chi_{B(0,N)}(\xi)\widehat{f}(\xi),$$
where $N>0, \,B(0,N)=\{\xi \in \mathbb{R}^{2}| \, |\xi|\leq N\}$ and
$\chi_{B(0,N)}$ is the characteristic function on $B(0,N)$. Also we
define
$$L^{2}_{N}\triangleq\{f\in L^{2}(\mathbb{R}^{2})|\, \mbox{ supp }\,\widehat{f}\subset B(0,N)\}.$$

\begin{proof}[{Proof of Proposition \ref{P1}}]
The first step is to consider the following approximate system
of (\ref{SQG}),
\begin{equation}\label{AMHD}
\left\{\aligned
&\partial_{t}\theta^{N}+\mathcal {J}_{N} (\mathcal {J}_{N}u^{N} \cdot \nabla \mathcal {J}_{N}\theta^{N})+ \Lambda_{x_{1}}^{2\alpha}\mathcal {J}_{N}\theta^{N}+\Lambda_{x_{2}}^{2\beta}\mathcal {J}_{N}\theta^{N}=0,\\
&u^{N}=\mathcal {R}^{\perp}\theta^{N},\\
&\theta^{N}(x,0)=\mathcal {J}_{N}\theta_{0}(x).
\endaligned\right.
\end{equation}
Using the Cauchy-Lipschitz theorem (Picard's Theorem, see \cite{Majdab}), we can find that for any fixed $N$, there exists a unique local solution $\theta^{N}$ on $[0,\,T_{N})$ in the functional setting $L^{2}_{N}$ with $T_{N}=T(N,\theta_{0})$. Due to $\mathcal {J}_{N}^{2}=\mathcal {J}_{N}$, we find that $\mathcal
 {J}_{N}\theta^{N}$ is also a solution to
(\ref{AMHD}) with the same initial data. According to the uniqueness,
we have
$$ \mathcal {J}_{N}\theta^{N}=\theta^{N}.$$
Consequently the approximate system (\ref{AMHD}) reduces to
\begin{equation}\label{AMHD1}
\left\{\aligned
&\partial_{t}\theta^{N}+\mathcal {J}_{N} ( u^{N} \cdot \nabla \theta^{N})+ \Lambda_{x_{1}}^{2\alpha} \theta^{N}+\Lambda_{x_{2}}^{2\beta} \theta^{N}=0,\\
&u^{N}=\mathcal {R}^{\perp}\theta^{N},\\
&\theta^{N}(x,0)=\mathcal {J}_{N}\theta_{0}(x).
\endaligned\right.
\end{equation}
By the basic energy estimate, we conclude that $\theta^{N}$ of \eqref{AMHD1} satisfies
\begin{align}
\|\theta^{N}(t)\|_{L^{2}}^{2}+2\int_{0}^{t}{
(\|\Lambda_{x_{1}}^{\alpha}\theta^{N}(\tau)\|_{L^{2}}^{2}
+\|\Lambda_{x_{2}}^{\beta}\theta^{N}(\tau)\|_{L^{2}}^{2})\,d\tau}\leq
\|\theta_{0}\|_{L^{2}}^{2}.\nonumber
\end{align}
Hence, the local solution can be extended into a global
one, by the standard Picard Extension Theorem (cf. \cite{Majdab}).
Moreover, the $H^s$-estimate allows us to derive
\begin{align}&\label{A12}
 \frac{d}{dt}\|\theta^{N}(t)\|_{H^{s}}^{2}
+\|\Lambda_{x_{1}}^{\alpha}\theta^{N}\|_{H^{s}}^{2}
+\|\Lambda_{x_{2}}^{\beta}\theta^{N}\|_{H^{s}}^{2}\nonumber\\
&\leq C(\|\nabla u^{N}\|_{L^{\infty}}+\|\nabla \theta^{N}\|_{L^{\infty}})
\|\theta^{N}\|_{H^{s}}^{2}\nonumber\\
&\leq C\|\theta^{N}\|_{H^{s}}^{3-
\frac{1}{2(1+\alpha)}-\frac{1}{2(1+\beta)}}
\|\Lambda_{x_{1}}^{\alpha}\theta^{N}\|_{H^{s}}^{\frac{1}{2(1+\alpha)}}
\|\Lambda_{x_{2}}^{\beta}\theta^{N}\|_{H^{s}}^{\frac{1}{2(1+\beta)}}\nonumber\\&\leq
\frac{1}{2}\|\Lambda_{x_{1}}^{\alpha}\theta^{N}\|_{H^{s}}^{2}
+\frac{1}{2}\|\Lambda_{x_{2}}^{\beta}\theta^{N}\|_{H^{s}}^{2}+
C\|\theta^{N}\|_{H^{s}}^{\frac{12(1+\alpha)(1+\beta)-2(2+\alpha+\beta)}
{4(1+\alpha)(1+\beta)-(2+\alpha+\beta)}},
\end{align}
where we have used the following facts (see \eqref{yrtew1} and \eqref{yrtew2})
\begin{align}\label{sdtf11}
\|\nabla \theta^{N}\|_{L^{\infty}} \leq&
C\|\nabla \theta^{N}\|_{L^{2}}^{1-
\frac{1}{2(1+\alpha)}-\frac{1}{2(1+\beta)}}
\|\Lambda_{x_{1}}^{\alpha}\Delta \theta^{N}\|_{L^{2}}^{\frac{1}{2(1+\alpha)}}
\|\Lambda_{x_{2}}^{\beta}\Delta \theta^{N}\|_{L^{2}}^{\frac{1}{2(1+\beta)}}
\nonumber\\ \leq&
C\|\theta^{N}\|_{H^{s}}^{1-
\frac{1}{2(1+\alpha)}-\frac{1}{2(1+\beta)}}
\|\Lambda_{x_{1}}^{\alpha}\theta^{N}\|_{H^{s}}^{\frac{1}{2(1+\alpha)}}
\|\Lambda_{x_{2}}^{\beta}\theta^{N}\|_{H^{s}}^{\frac{1}{2(1+\beta)}},
\end{align}
\begin{align}\label{sdtf12}
\|\nabla u^{N}\|_{L^{\infty}}
 \leq& C\|\nabla \theta^{N}\|_{L^{2}}^{1-
\frac{1}{2(1+\alpha)}-\frac{1}{2(1+\beta)}}
\|\Lambda_{x_{1}}^{\alpha}\Delta \theta^{N}\|_{L^{2}}^{\frac{1}{2(1+\alpha)}}
\|\Lambda_{x_{2}}^{\beta}\Delta \theta^{N}\|_{L^{2}}^{\frac{1}{2(1+\beta)}}
\nonumber\\ \leq&
C\|\theta^{N}\|_{H^{s}}^{1-
\frac{1}{2(1+\alpha)}-\frac{1}{2(1+\beta)}}
\|\Lambda_{x_{1}}^{\alpha}\theta^{N}\|_{H^{s}}^{\frac{1}{2(1+\alpha)}}
\|\Lambda_{x_{2}}^{\beta}\theta^{N}\|_{H^{s}}^{\frac{1}{2(1+\beta)}}.
\end{align}
Denoting
$$\gamma:=\frac{8(1+\alpha)(1+\beta)-(2+\alpha+\beta)}
{4(1+\alpha)(1+\beta)-(2+\alpha+\beta)}>1,$$
we thus get from (\ref{A12}) that
\begin{align}
 \frac{d}{dt}\|\theta^{N}(t)\|_{H^{s}} \leq \widetilde{C} \|\theta^{N}\|_{H^{s}}^{\gamma},\nonumber
\end{align}
where $\widetilde{C}>0$ is an absolute constant.
One observes that for all $N$
$$
\sup_{0\leq t\leq T}\|\theta^{N}(t)\|_{H^{s}}
 \leq \frac{\|\theta_{0}\|_{H^{s}}}{\Big[1-(\gamma-1)\widetilde{C}T\|\theta_{0}
 \|_{H^{s}}\Big]^{\frac{1}{\gamma-1}}},\qquad T<\frac{1}{(\gamma-1)\widetilde{C}
\|\theta_{0}\|_{H^{s}}}.$$
Therefore, $\theta^{N}$ is uniformly bounded in $C([0, T]; H^{s})$ with $s\geq2$.
We may deduce that
$$ \partial_{t}\theta^{N}\in L_{t}^{\infty}
([0, T]);\,H_{x}^{-\sigma}(\mathbb{R}^{2})\quad \mbox{for some} \,\,
\sigma\geq 2.$$
Since the embedding $L^{2}\hookrightarrow H^{-\sigma}$ is locally compact, the well-known Aubin-Lions argument allows
us to conclude that, up to extraction, subsequence
$\{\theta^{N}\}_{N\in\mathbb{N}}$ satisfies
$$ \|\theta^{N}-\theta^{N'}\|_{L^{2}}\rightarrow0,\quad as\quad N,\,\,N'\rightarrow\infty.$$
Thanks to the interpolation ($\|f\|_{H^{s'}}\leq C
\|f\|_{L^{2}}^{1-\frac{s'}{s}}\|f\|_{H^{s}}^{\frac{s'}{s}}$ for any
$s'<s$), we deduce that
$$ \|\theta^{N}-\theta^{N'}\|_{H^{s'}}\rightarrow0,\quad
as\quad N,\,\,N'\rightarrow\infty.$$
This implies the strong convergence limit $\theta\in C([0, T];
H^{s'})$ for any $s'<s$.
Therefore, this is enough for us to show that up to
extraction, sequence $\{\theta^{N}\}_{N\in\mathbb{N}}$ has a limit
$\theta$ satisfying
\begin{equation}\label{AMHD11}
\left\{\aligned
&\partial_{t}\theta+(u\cdot \nabla)\theta+ \Lambda_{x_{1}}^{2\alpha} \theta +\Lambda_{x_{2}}^{2\beta} \theta =0,\\
&u=\mathcal {R}^{\perp}\theta,\\
& \theta(x,0)=\theta_{0}(x).
\endaligned\right.
\end{equation}
Moreover, we have $\theta\in L^{\infty}([0, T];
H^{s}(\mathbb{R}^{2}))$. Finally, we begin to show the time continuity of the solution in $H^{s}(\mathbb{R}^{2})$, namely,
 \begin{align}\label{asaqwwe56}
 \theta\in C([0, T]; H^{s}(\mathbb{R}^{2})).
  \end{align}
Based on the above argument, we first have
 \begin{align} \sup_{0\leq t\leq T}\|\theta\|_{H^{s}}<\infty.
 \nonumber
 \end{align}
By the equivalent norm, it yields
 \begin{align}\label{t2.03}
 \|\theta(t_{1})-\theta(t_{2})\|_{H^{s}}=\Big\{(\sum_{k<N}+ \sum_{k\geq N})
 (2^{ks}\|\Delta_{k}\theta(t_{1})-\Delta_{k}\theta(t_{2})\|_{L^{2}})^{2}
 \Big\}^{\frac{1}{2}}.
 \end{align}
Let $\varepsilon>0$ be arbitrarily small. Due to $\theta\in
L^{\infty}([0, T]; H^{s}(\mathbb{R}^{2}))$, there exists an integer
$M=M(\varepsilon)>0$ such that
 \begin{align}\label{t2.04}
 \Big\{\sum_{k\geq M}
 (2^{ks}\|\Delta_{k}\theta(t_{1})-\Delta_{k}\theta(t_{2})\|_{L^{2}})^{2}
 \Big\}^{\frac{1}{2}}<\frac{\varepsilon}{2}.
 \end{align}
By the first equation of \eqref{AMHD11}, it yields
\begin{align}
\Delta_{k}\theta(t_{1})-\Delta_{k}\theta(t_{2}) =&\int_{t_{1}}^{t_{2}}{\frac{d}{d\tau}
\Delta_{k}\theta(\tau)\,d\tau} 
 = -\int_{t_{1}}^{t_{2}}{ \Delta_{k} [(u\cdot\nabla)\theta+ \Lambda_{x_{1}}^{2\alpha} \theta +\Lambda_{x_{2}}^{2\beta} \theta](\tau)\,d\tau}.\nonumber
\end{align}
Thus, using the Bernstein lemma and $s>1$, we conclude that
\begin{align}
& \sum_{k<M}
 2^{2ks}\|\Delta_{k}\theta(t_{1})-\Delta_{k}\theta(t_{2})\|_{L^{2}}^{2}\nonumber\\
 &= \sum_{k<M}
 2^{2ks}\Big(\Big\|\int_{t_{1}}^{t_{2}}{ \Delta_{k} [ (u\cdot\nabla) \theta+ \Lambda_{x_{1}}^{2\alpha} \theta +\Lambda_{x_{2}}^{2\beta} \theta](\tau)\,d\tau}\Big\|_{L^{2}}\Big)^{2}\nonumber\\
&\leq \sum_{k<M}
 2^{2ks}\Big(\int_{t_{1}}^{t_{2}}{ \|\Delta_{k}[(u\cdot\nabla) \theta+ \Lambda_{x_{1}}^{2\alpha} \theta +\Lambda_{x_{2}}^{2\beta} \theta]\|_{L^{2}}(\tau)\,d\tau}\Big)^{2}
\nonumber\\
&\leq \sum_{k<M}
 2^{2ks}\Big(\int_{t_{1}}^{t_{2}}{ [\|\Delta_{k}\nabla\cdot(u\otimes \theta)\|_{L^{2}}+
\| \Delta_{k}\Lambda_{x_{1}}^{2\alpha} \theta\|_{L^{2}}
+\| \Delta_{k}\Lambda_{x_{2}}^{2\beta} \theta\|_{L^{2}}](\tau)\,d\tau}\Big)^{2} \nonumber\\
&\leq\sum_{k<M}
 2^{2k}\Big(\int_{t_{1}}^{t_{2}}{  2^{ks}\|\Delta_{k} (u\otimes \theta) (\tau)\|_{L^{2}} \,d\tau}\Big)^{2}
\nonumber\\
&\quad+ \sum_{k<M}
 2^{4\alpha k}\Big(\int_{t_{1}}^{t_{2}}{
2^{ks}\|\Delta_{k}\theta(\tau)\|_{L^{2}}\,d\tau}\Big)^{2}
+ \sum_{k<M}
 2^{4\beta k}\Big(\int_{t_{1}}^{t_{2}}{
2^{ks}\|\Delta_{k}\theta(\tau)\|_{L^{2}}\,d\tau}\Big)^{2}
\nonumber\\
&\leq\sum_{k<M}
 2^{2k}\Big(\int_{t_{1}}^{t_{2}}{ \|u\theta(\tau)\|_{H^{s}} \,d\tau}\Big)^{2}
+ \sum_{k<M}
 2^{4\alpha k}\Big(\int_{t_{1}}^{t_{2}}{
\|\theta(\tau)\|_{H^{s}}\,d\tau}\Big)^{2}
\nonumber\\
&\quad+ \sum_{k<M}
 2^{4\beta k}\Big(\int_{t_{1}}^{t_{2}}{
\|\theta(\tau)\|_{H^{s}}\,d\tau}\Big)^{2}
\nonumber\\
&\leq C\sum_{k<M}
 2^{2k} \|u\theta\|_{L_{t}^{\infty}H^{s}}^{2}|t_{1}-t_{2}|^{2}+
 C\sum_{k<M}
 (2^{4\alpha k}+2^{4\beta k})\|\theta\|_{L_{t}^{\infty}H^{s}} |t_{1}-t_{2}|^{2}
 \nonumber\\
&\leq C\sum_{k<M}
 2^{2k}|t_{1}-t_{2}|^{2}\Big(\|u\|_{L_{t}^{\infty}H^{s}}^{2}
 \|\theta\|_{L_{t}^{\infty}L^{\infty}}^{2}+\|\theta\|_{L_{t}^{\infty}H^{s}}^{2}
 \|u\|_{L_{t}^{\infty}L^{\infty}}^{2}\Big)\nonumber\\&\quad
 +C\sum_{k<M}
 (2^{4\alpha k}+2^{4\beta k})\|\theta\|_{L_{t}^{\infty}H^{s}} |t_{1}-t_{2}|^{2}\nonumber\\
&\leq  C
 (2^{2M}+2^{4\alpha M}+2^{4\beta M})|t_{1}-t_{2}|^{2}\Big(\|\theta\|_{L_{t}^{\infty}H^{s}}^{4}+
 \|\theta\|_{L_{t}^{\infty}H^{s}} \Big).\nonumber
\end{align}
Therefore, the following one holds true
 \begin{align}\label{t2.07}
 \Big\{\sum_{k< M}
 (2^{ks}\|\Delta_{k}\theta(t_{1})-\Delta_{k}\theta(t_{2})\|_{L^{2}})^{2}
 \Big\}^{\frac{1}{2}}<\frac{\varepsilon}{2}
 \end{align}
as long as $|t_{1}-t_{2}|$ is small enough. Combining (\ref{t2.03}), (\ref{t2.04}) with (\ref{t2.07}) yields \eqref{asaqwwe56}, namely $\theta\in C([0, T]; H^{s}(\mathbb{R}^{2})$. Thanks to \eqref{sdtf11} and \eqref{sdtf12}, we have
\begin{eqnarray*}  \int_{0}^{t}(\|\nabla u(\tau)\|_{L^{\infty}}+\|\nabla \theta(\tau)\|_{L^{\infty}})\,d\tau<\infty,
\end{eqnarray*}
which leads to the uniqueness immediately. Therefore, this completes the proof of Proposition \ref{P1}.
\end{proof}

\vskip .2in
\section*{Acknowledgements}
The author is supported by the National Natural Science Foundation of China (No. 11701232) and the Natural Science Foundation of Jiangsu Province (No. BK20170224).

\end{document}